\documentclass[a4paper,12pt]{amsart}
\usepackage{amsmath,amssymb,amsfonts,graphicx}
\usepackage{amsmath,amssymb,amsfonts,graphicx,color,fancyhdr}
\usepackage[latin1]{inputenc}
\usepackage{psfrag}
\usepackage{comment}
\usepackage{enumerate}

\newtheorem{theorem}{Theorem}[section]
\newtheorem{proposition}[theorem]{Proposition}
\newtheorem{lemma}[theorem]{Lemma}
\newtheorem{definition}[theorem]{Definition}

\newtheorem{remark}[theorem]{Remark}
\newtheorem{corollary}[theorem]{Corollary}

\numberwithin{equation}{section}

\usepackage[colorlinks=true, pdfstartview=FitV, linkcolor=blue, citecolor=blue, urlcolor=blue]{hyperref}
\usepackage{psfrag,caption}

\allowdisplaybreaks

\def\be#1 {\begin{equation} \label{#1}}
	\newcommand{\ee}{\end{equation}}

\def\sqw{\hbox{\rlap{\leavevmode\raise.3ex\hbox{$\sqcap$}}$%
		\sqcup$}}
\def\findem{\ifmmode\sqw\else{\ifhmode\unskip\fi\nobreak\hfil
		\penalty50\hskip1em\null\nobreak\hfil\sqw
		\parfillskip=0pt\finalhyphendemerits=0\endgraf}\fi}

\newcommand{\R}{{\mathbb {R}}}

\newcommand{\He}{{\mathbb H}}
\newcommand{\Hep}{{\mathbb H_\varepsilon^n}}
\newcommand{\N}{{\mathbb N}}

\newcommand{\C}{{\mathbb C}}

\newcommand{\vep}{\varepsilon}

\title[Fourier multipliers on  Sobolev spaces]
{Homogeneous Fourier and Weyl multipliers on \\ 
Sobolev spaces related to the Heisenberg group}

\author[R. Basak]{Riju Basak}

\author[R. Garg]{Rahul Garg}
\address{Department of Mathematics, Indian Institute of Science Education and Research Bhopal, 462 066, India.}
\email{\{riju.basak, rahulgarg\}@iiserb.ac.in}

\author[S. Thangavelu]{Sundaram Thangavelu}
\address{Department of Mathematics, Indian Institute of Science, Bangalore, 560 012, India.}
\email{veluma@iisc.ac.in}


\begin{document}

\begin{abstract}
Inspired by the work of A. Bonami and S. Poornima that a non-constant function which is homogeneous of degree $0$ cannot be a Fourier multiplier on  homogeneous Sobolev spaces, we establish analogous results for Fourier multipliers on the Heisenberg group $ \He^n $ and Weyl multipliers on $ \C^n $ acting on  Sobolev Spaces. 
\end{abstract}

\maketitle

\tableofcontents
\section{Introduction and the main results} \label{sec-intro}
Our main goal in this article is to establish an analogue of an interesting theorem of Bonami and Poornima \cite{BON-POOR} on Fourier multipliers on Sobolev spaces. Let $ X $ be a translation invariant Banach space of functions on $ \R^n$ such as $ L^p(\R^n) $ and Sobolev spaces $ W^{N,p}(\R^n).$ By a Fourier multiplier operator on $ X $ we mean  a translation invariant bounded linear operator $ T $  defined on  $ X .$ It is well known that to any such operator is associated a bounded measurable function $ m $ on $ \R^n$ such that $ \widehat{Tf}(\xi) = m(\xi) \widehat{f}(\xi) $ on $ X \cap L^2(\R^n)$ where $ \widehat{f} $ stands for the Fourier transform of $ f $ on $ \R^n.$ Consequently, such operators are usually denoted by $ T_m $ and we call the bounded function $ m $ the multiplier corresponding to $ T_m.$ The space of all such functions $m$ is said to be the Fourier multiplier space of $X$, with multiplier norm of $m$ identified with the operator norm of $T_m.$  By Plancherel theorem for the Fourier transform it is clear that $ T_m $ is bounded on $ L^2(\R^n) $ if and only if $ m \in L^\infty(\R^n).$ However, such an operator $ T_m $ need not be bounded on  $ L^p(\R^n)$ for $ p \neq 2 $ unless further assumptions are made on the multiplier $ m.$ Finding necessary and sufficient conditions on $ m $ so that $ T_m $ extends to $ L^p(\R^n) $ has a long history and the literature is quite large.\\

Given a Fourier multiplier $ T_m $ on $ L^p(\R^n), 1 \leq p \leq \infty $ it is natural to ask if it is also a bounded linear operator on the Sobolev space $ W^{N,p}(\R^n) $ consisting of all $ f \in L^p(\R^n) $ whose distributional derivatives $ \partial^\alpha f $ also belong to $ L^p(\R^n) $ for all $ |\alpha|\leq N.$ In \cite{Poornima}, it was shown that for $1 < p<\infty$, the Fourier multiplier space of any Sobolev space $W^{N,p}(\mathbb{R}^{n})$ is the same as that of $L^p(\mathbb{R}^{n})$ with norm equivalence, and that the analogous result for $p=1$ is true only for functions of one real-variable. They have  established  that the Fourier multiplier space of $W^{N,1}(\mathbb{R}^{n})$ is strictly larger than that of $L^1(\mathbb{R}^{n})$ when $n \geq 2$. Subsequently, Bonami and Poornima \cite{BON-POOR} showed the absence of certain natural operators of the singular integral type from the Fourier multiplier space of the Sobolev spaces $W^{N,1}(\mathbb{R}^{n})$.\\

In this connection they established the following result. Let $\mathring{W}^{N,p}(\mathbb{R}^{n})$ stand for the homogeneous Sobolev space consisting of all tempered distributions $f$ for which $ \partial^\alpha f \in L^p(\R^n) $ whenever $ |\alpha| = N.$  Recall that  a function $ m $ is said to be homogeneous of degree $ 0 $ if $ m( \lambda \xi) = m(\xi) $ for all $ \lambda >0.$  Such a function is completely determined by its restriction on the unit sphere.

\begin{theorem}[Bonami-Poornima] \label{thm-BP} 
Let $ m $ be a homogeneous function of degree $0.$  Then $ T_m $ is bounded on the homogeneous Sobolev space $\mathring{W}^{N,1}(\mathbb{R}^{n})$ if and only if $ m $ is  constant.
\end{theorem} 

In their work Bonami and Poornima have  proved that the Fourier multiplier space of $W^{N,1}(\mathbb{R}^{n})$ is contained in that of $\mathring{W}^{N,1}(\mathbb{R}^{n})$ and hence in the above theorem we can also replace $\mathring{W}^{N,1}(\mathbb{R}^{n})$ by $ W^{N,1}(\R^n).$ Recently, Kazaniecki and Wojciechowski \cite{Kaz-Woj} gave a necessary condition on multipliers of $\mathring{W}^{1,1}(\mathbb{R}^{n})$. They proved that the Fourier multipliers on $\mathring{W}^{1,1}(\mathbb{R}^{n})$ are necessarily continuous on $\mathbb{R}^{n}$. \\

In this article we are interested in proving analogues of  the above theorem of Bonami and Poornima for the  Fourier transform on the Heisenberg group $ \He^n $ and also for the closely related Weyl transform on $ \C^n.$ In order to state our results, we need to set up some notation. Unlike the Fourier transform on $ \R^n $, the Fourier transform on $ \He^n $ is operator valued. More precisely, if $ f $ is an integrable function on $ \He^n$ then its Fourier transform is the operator valued function $ \lambda \rightarrow \widehat{f}(\lambda)  \in B(L^2(\R^n)) $ from $ \R^\ast $ into the Banach space of bounded linear operators on $ L^2(\R^n).$ Given a bounded function $ m(\lambda) $ on $ \R^\ast $ taking values in $ B(L^2(\R^n)) $ we define an operator $ T_m $ on $ L^2(\He^n) $ by $ \widehat{T_mf}(\lambda) = m(\lambda)\widehat{f}(\lambda).$ These operators which are clearly bounded on $ L^2(\He^n)$ are called (left) Fourier multipliers for obvious reasons. They commute with right translations on $ \He^n $ and it can be shown that they give all  right-invariant bounded operators  on $ L^2(\He^n).$ When they extend to $ L^p(\He^n) $ as bounded operators, we will call them (and also the functions $ m$) $ L^p$ multipliers. There is a vast literature giving sufficient conditions on the multiplier $ m $ so that $ T_m $ is bounded on $ L^p(\He^n),$ see e.g. \cite{MS}. \\

As in the case of $ \R^n $, Fourier multipliers on $ \He^n$ can also be realised as convolution operators. Indeed, if $ T_m $ is right-invariant and bounded on $ L^p(\He^n), 1 \leq p < \infty $ then there exists a pseudo-measure $ S $ on $ \He^n $ such that $ T_m f = S \ast f, $  see e.g. \cite{RV}. As in the case of $ \R^n $, the multiplier $ m $ is related to $ S $ via Fourier transform. In the Euclidean case this relation reads as $ m(\xi) = \widehat{S}(\xi) $ and we observe that the function $ m $ is homogeneous of degree zero if and only if the pseudo-measure $ S $ is homogeneous of degree $ -n.$ It is therefore natural to make the following definition: we say that the Fourier multiplier $ m $ on $ \He^n $ is homogeneous of degree zero if the associated pseudo-measure is homogeneous of degree $ -Q $ where $ Q = 2n+2 $ is the homogeneous dimension of $ \He^n.$ Here, the homogeneity of $ S $ is defined in terms of the non-isotropic dilations $ \delta_r.$ As $ \He^n = \C^n \times \R $ as a manifold, we define $ \delta_r(z,t) = (rz,r^2 t), r>0.$ Thus $ m $ is homogeneous of degree zero if $ \delta_r S = r^{-Q} S.$ In case $ S $ is given by a function $ K $ the above simply means $ K(rz,r^2t) = r^{-Q} K(z,t) $. \\

We can also state the above definition in terms of the multiplier $ m.$ In the Euclidean case, $ m $ is homogeneous of degree zero if and only if $ T_m $ commutes with the standard dilations $ d_r, r>0.$ In a similar way, we can say that the Fourier multiplier $ m $ on $\He^n$ is homogeneous of degree zero if $ T_m $ commutes with the non-isotropic dilations $ \delta_r.$ It can be checked that these two definitions coincide. The connection between the (group) Fourier transforms of $ f $ and $ \delta_r f(z,t) = f(\delta_r(z,t)) $ for a function $ f $ on $ \He^n $ is given by
$$ \widehat{(\delta_rf)}(\lambda) = r^{-Q} d_r  \circ \widehat{f} \left( \lambda / r^2 \right) \circ d_r^{-1} $$
for every $ \lambda \in \R^* $ and $ r >0. $  
From this we see that $ T_m $ commutes with $ \delta_r $ if and only if we have the relation
$ d_r \circ m \left( \lambda / r^2 \right) \circ d_r^{-1} = m(\lambda).$
Thus we see that $ m(\lambda) $ is completely determined by $ m(1)$ and $ m(-1).$ Indeed, 
\begin{equation} \label{hom} 
m(\lambda)  = d_{\sqrt{|\lambda|}} \circ m \left( \lambda / |\lambda| \right) \circ d_{\sqrt{|\lambda|}}^{-1} . 
\end{equation}
As a consequence, the Fourier multiplier $ m $ is homogeneous of degree zero if and only if it satisfies relation \eqref{hom}.\\

The special orthogonal group $ SO(n) $ acts on Fourier multipliers $ m $ on $ \R^n $ by $ \rho_\sigma m(\xi) = m(\sigma^{-1} \xi) $ and we have the relation
$$  \rho_{\sigma^{-1}}\circ T_m \circ  \rho_\sigma  f = T_{\rho_\sigma m}f .$$ If $ \pi_k $ is the irreducible unitary representation of $ SO(n) $ acting on the space $ \mathcal{H}_k $ of spherical harmonics of degree zero, then
integrating the above relation against the character $ \chi_k $ of $ \pi_k $ we have
\begin{equation}\label{eq-1}  \int_{SO(n)}  \chi_k(\sigma^{-1})\rho_{\sigma^{-1}}\circ T_m \circ \rho_\sigma f \, d\sigma =  T_{m_k}f \end{equation}
where $ m_k $ is the projection of $ m $ into $ \mathcal{H}_k $ given by 
$$ m_k(\xi) = \int_{SO(n)} \chi_k(\sigma^{-1})  m(\sigma^{-1}\xi) \, d\sigma .$$
This function $ m_k $  has the  explicit representation  given by 
\begin{equation}\label{rep-euc}  m_k(\xi) = \sum_{j=1}^{d_k} c_{k,j} P^k_j(\xi) |\xi|^{-k} 
\end{equation}
where $ P^k_j(\xi) $ are solid harmonics of degree $ k $ so that their restrictions $ Y^k_j(\omega) $ to $ S^{n-1} $ form a basis for $ \mathcal{H}_k.$ Observe that $ T_{m_k} $ is a linear combination of the higher order Riesz transforms  $ R_{k,j} = P^k_j(\partial) (-\Delta)^{-k/2} $ with multipliers $ Y^k_j(\omega).$ The above expression for $ m_k $ plays an important role in the proof of Theorem \ref{thm-BP}. \\

In the case of Heisenberg group, we have the action of the unitary group $U(n) $ given by $ R_\sigma (z,t) = (\sigma^{-1}z,t) $ which gives rise to an action on functions. It is known that 
$$ \widehat{R_\sigma f}(\lambda) = \mu_\lambda(\sigma) \widehat{f}(\lambda) \mu_\lambda(\sigma)^\ast $$
where $ \mu_\lambda(\sigma) $ are certain unitary operators on $ L^2(\R^n).$ As in the Euclidean case we have the following relation for Fourier multipliers on $ \He^n$: 
$$ R_{\sigma}\circ T_m \circ R_{\sigma^{-1}} f = T_{R_\sigma m}f $$
where $ R_\sigma m(\lambda) = \mu_\lambda(\sigma) m(\lambda) \mu_\lambda(\sigma)^\ast .$ Let $ \delta = \delta_{a,b} $ be a class one representation of $ U(n) $ realised on $ \mathcal{H}_\delta ,$ the space of bi-graded spherical harmonics of bi-degree $ (a,b).$ Let $ \chi_\delta $ stand for the character of $ \delta $ and define
$$  m_\delta(\lambda)  = \int_{U(n)}  \chi_{\delta}(\sigma^{-1}) \mu_\lambda(\sigma) m(\lambda) \mu_\lambda(\sigma)^\ast \, d\sigma .$$ Then we have the following analogue of  (\ref{eq-1}):
\begin{equation}\label{eq-2} 
\int_{U(n)}  \chi_\delta(\sigma^{-1}) R_{\sigma}\circ T_m \circ R_{\sigma^{-1}} f \, d\sigma =  T_{m_\delta} f . \\
\end{equation}

As observed by Geller \cite{Geller}, the role of spherical harmonics $ Y^k_j $ or rather the solid harmonics $ P^k_j $ on $ \R^n $ are played by the Weyl correspondences $ G_\lambda(P^\delta_j) $ of solid harmonics $ P^\delta_j $ coming from a basis of $ \mathcal{H}_\delta.$ We refer to  subsections \ref{subsec-Heisen} and \ref{subsec-sph-harmonics} for the definition and properties of these unbounded operators. We let $ H(\lambda) = -\Delta+\lambda^2 |x|^2 $ denote the scaled Hermite operator whose spectrum consists of points of the form $ (2k+n)|\lambda|, k \in \N.$  The operators $ G_\lambda(P^\delta_j)H(\lambda)^{-(a+b)/2} $ are uniformly bounded on $ L^2(\R^n) $ and define Fourier multipliers on $ L^2(\He^n) $ which are  the higher order Riesz transforms on the Heisenberg group. \\

We are interested in the boundedness of $ T_m $ on the Sobolev spaces $ W^{N,1}(\He^n) $ or the homogeneous Sobolev spaces $ \mathring{W}^{N,1}(\He^n) $ on the Heisenberg group. These Sobolev spaces are defined in terms of certain left invariant vector fields which are the counter parts of partial derivatives on $ \R^n.$ It has been proved in \cite{Jit-Radha-V} that the class of Fourier multipliers on the Sobolev spaces $ W^{N,p}(\He^n) $ coincides with the class of Fourier multipliers  on $ L^p(\He^n) $ for $ 1 < p < \infty.$ They have also obtained an abstract characterisation of Fourier multipliers on $ W^{N,1}(\He^n).$ Consider now the following condition on the multiplier $m$: 
\begin{equation}\label{rep-hei} 
m_\delta(\lambda) = \sum_{j=1}^{d(\delta)} c_{\delta,j}(\lambda) G_\lambda(P^\delta_j) H(\lambda)^{-(a+b)/2}. 
\end{equation}
It is clear from the work of Geller that the eigenspaces of $ H(\lambda) $ corresponding to the eigenvalues $ (2k+n)|\lambda| $ are analogues of spheres of radius $ \sqrt{(2k+n)|\lambda|} .$ In view of this, it is clear that  (\ref{rep-hei}) is the analogue of (\ref{rep-euc}). Unlike the Euclidean case, (\ref{rep-hei}) is not automatically satisfied under the assumption that $ m $ is homogeneous of degree zero. In view of this remark the following result is the exact analogue of Theorem \ref{thm-BP}.

\begin{theorem}\label{thm2} 
Assume that the (left) Fourier multiplier $m,$ which is homogeneous of degree zero, satisfies the condition (\ref{rep-hei}) for any class one representation $ \delta $ of $ U(n).$ If the corresponding operator $ T_m $ is bounded either on $ W^{N,1}(\He^n)$ or $ \mathring{W}^{N,1}(\He^n),$ then $ m(\lambda) = c_1\chi_{\R^+}(\lambda) I +c_2\chi_{\R^+}(-\lambda) I.$
\end{theorem}

Unlike the Euclidean case, we  do not know if every $ W^{N,1}(\He^n) $ multiplier is automatically a $ \mathring{W}^{N,1}(\He^n) $ multiplier. Hence, separate proofs have to be given for each case. It is still an open question if the above theorem is true without the extra assumption (\ref{rep-hei}) on the multiplier $ m.$ We plan to address this question elsewhere. \\

In this article, we also study  the boundedness of Weyl multipliers on $L^1(\mathbb{C}^{n})$ and on Laguerre Sobolev spaces $W^{N,1}_{L}(\mathbb{C}^{n})$ and prove an analogue of  Theorem \ref{thm-BP}  in the context of Weyl multipliers. \\

First we need to recall some definitions briefly, postponing more details to later sections. Let $W$ denote the Weyl transform which takes functions on  $\mathbb{C}^{n}$ into bounded linear operators on $ L^2(\R^n).$  The Weyl transform $ W $ or more generally the family of Weyl transforms $ W_\lambda $ are closely related to the Fourier transform on $ \He^n,$ see subsection \ref{subsec-Heisen}. $ W $  shares many properties with the Fourier transform: e.g. there is a Plancherel theorem for the Weyl transform: for $ f \in L^2(\C^n) $ its Weyl transform is a Hilbert-Schmidt operator and $ \| W(f)\|_{HS} = \|f\|_2.$ Consequently, given a bounded linear operator $ M $ on $ L^2(\R^n) $ we can define an operator $ T_M $ on $ L^2(\C^n) $ by the prescription $ W(T_Mf) = MW(f).$ In view of the Plancherel theorem it is clear that $ T_M $ is a bounded operator on $ L^2(\C^n).$ Such operators are called Weyl multipliers and $ M $ is called the (left) Weyl multiplier. If $ T_M $ initially defined on $ L^p \cap L^2(\C^n) $ extends to $ L^p(\C^n) $ as a bounded operator we say that $ M $ is an $ L^p $ Weyl  multiplier. Similar definition applies to other function spaces $ X $ on $ \C^n $ which are required to be (twisted) translation invariant. \\

In \cite{Mauceri}, Mauceri initiated  the study of Weyl multipliers on Lebesgue spaces. He obtained some sufficient conditions on a bounded linear operator $M$ on $L^p(\mathbb{R}^n)$ for the Weyl multiplier $T_M$ to be bounded on $L^p(\mathbb{C}^n)$. Later, using Littlewood-Paley-Stein theory, the third author obtained a sharper result when $ M = m(H) $ is a function of the Hermite operator, see \cite{Thangavelu-LPS}. We also refer to the recent works of Bagchi-Thangavelu \cite{Bagchi-Thangavelu-1} and \cite{Bagchi-Thangavelu-2}. In \cite{Radha-Thangavelu}, Radha and Thangavelu took up the problem of the Weyl multipliers on Laguerre-Sobolev spaces $ W_L^{N,p}(\C^n).$ Here $ L $ stands for the special Hermite operator, also called the twisted Laplacian which plays the role of $ - \Delta $ in the context of Weyl multipliers. The Laguerre Sobolev spaces are defined in terms of certain vector fields associated to $ L$, see subsection \ref{subsec-Laguerre-Sob} for details. Analogous to the work of Poornima \cite{Poornima}, it was shown in \cite{Radha-Thangavelu} that for $1 < p<\infty$, the space of Weyl multipliers of any Laguerre Sobolev space $W^{N,p}_{L}(\mathbb{C}^{n})$ coincides with that of $L^p(\mathbb{C}^{n})$ with norm equivalence. They also characterised the space of Weyl multipliers of $W^{N,1}_{L}(\mathbb{C}^{n})$, showing it to be the dual of certain function space. \\

In the present article, we address the case of $p=1$ for Weyl multipliers, in the spirit of the work done in \cite{BON-POOR}. Typical examples of  Fourier multipliers which are bounded on $ L^p(\R^n) $ for all $ 1 < p < \infty $ but  not bounded on $ L^1(\R^n) $ are provided by the Riesz transforms $ R_j $ defined by $ \widehat{R_j f}(\xi) = i \frac{\xi_j}{|\xi|} \widehat{f}(\xi).$  As the associated multipliers $ m_j(\xi) = i \frac{\xi_j}{|\xi|} $ are non-constant functions which are homogeneous of degree zero, they are certainly unbounded on $ L^1(\R^n) $ and $ W^{N,1}(\R^n) $ and it follows from the theorem of Bonami and Poornima that $ R_j $ are also not bounded on any of the homogeneous Sobolev spaces $ \mathring{W}^{N,1}(\R^n).$  \\

Analogues of the Riesz transforms $ R_j $ in the context of the Weyl multipliers are given by the multipliers $ M_j = A_j H^{-1/2} $ where $ A_j = \frac{\partial}{\partial x_j}+x_j $ are the annihilation operators and $ H = H(1) = -\Delta+|x|^2 $ is the simple harmonic oscillator or the Hermite operator. It is known that the operators $ T_{M_j} $ are bounded on $ L^p(\C^n) $ for $ 1 < p \leq \infty $ but not on $ L^1(\C^n).$ In this paper we will show that they are also unbounded on the Laguerre-Sobolev spaces $ W_L^{N,1}(\C^n).$ \\

The operators $ T_{M_j} $ are indeed the Riesz transforms associated to the special Hermite operator and their unboundedness on the Sobolev spaces $ W_L^{N,1}(\C^n)$ will be used to prove an analogue of Theorem \ref{thm-BP} for Weyl multipliers. In order to formulate our result, we need a  notion of homogeneity for bounded operators on $ L^2(\R^n).$ \\

Given a bi-graded solid harmonic $ P $ we let $ G(P) = G_1(P) $ stand for the Weyl correspondence associated to $ P.$  For any class one representation $ \delta = \delta_{a,b} $ of $ U(n) $ realised on the Hilbert space $ \mathcal{H}_\delta  \subset L^2(S^{2n-1}) $ we choose an orthonormal basis $ Y^\delta_j , j =1, 2, \ldots, d(\delta) $ which are restrictions of the solid harmonics  $ P^\delta_j $ to the unit sphere. Let $ E_k $ stand for the finite dimensional subspace of $ L^2(\R^n) $ consisting precisely of eigenfunctions of $ H $ with eigenvalue $ (2k+n).$  Let $ P_k: L^2(\R^n) \rightarrow E_k $ be the orthogonal projection. Geller \cite{Geller} has shown that the family $ G(P^\delta_j), k \geq a, j = 1, 2, \ldots, d(\delta) $, where $ \delta = \delta_{a,b} $ runs along all class one representations of $ U(n) $, is an orthogonal basis for the space  of all bounded linear operators taking $ E_k $ into $ L^2(\R^n).$ By suitable choice of constants we can make $ C_\delta(2k+n)^{-1} G(P^\delta_j) $ form an orthonormal basis. More generally, the collection 
$$ S^\delta_{j,k} =  C_\delta(2k+n)^{-1} G(P^\delta_j) P_k $$ forms an orthonormal basis for $ \mathcal{S}_2 \subset B(L^2(\R^n)) $ which is the Hilbert space of Hilbert-Schmidt operators on $ L^2(\R^n) $ equipped with the inner product $ (T,S) = tr(S^\ast T).$\\

We now consider the following expansion for any $ M \in B(L^2(\R^n)).$  The Hilbert-Schmidt operator $ MP_k $ can be expanded as
\begin{equation} \label{MPk-Geller}
MP_k = \sum_{\delta} \sum_{j=1}^{d(\delta)} C_\delta(2k+n)^{-2} (MP_k, G(P^\delta_j)P_k) G(P^\delta_j) P_k.
 \end{equation} 
 A simple  calculation making use of the  orthonormality of $ S^\delta_{j,k} $ 
 leads us to 
\begin{equation} 
\|MP_k\|_{HS}^2 = \sum_{\delta} \sum_{j=1}^{d(\delta)} C_\delta(2k+n)^{-2} |(MP_k, G(P^\delta_j)P_k)|^2. 
\end{equation} 
Since $ MP_k $ is the operator analogue of `restricting a function to a sphere in $ \R^n$', we would like to impose the condition that $ \|MP_k\|_{HS}^2 $ is independent of $k.$ This is guaranteed, for example, if  the coefficients in the above expansion of the operator $ M $  satisfy the condition
$ C_\delta(2k+n)^{-1} (MP_k, G(P^\delta_j)P_k) = B_{\delta,j} . $ If that is the case the expansion \eqref{MPk-Geller} reduces to
\begin{equation} \label{op-hom}
MP_k = \sum_{\delta} \sum_{j=1}^{d(\delta)} B_{\delta,j} C_\delta(2k+n)^{-1} G(P^\delta_j) P_k. 
\end{equation} 
We say that $ M $ is homogeneous of degree zero if the expansion \eqref{op-hom} holds for every $ k.$ (We refer to subsection \ref{subsec-hom-degree-0} for further discussions of this definition.) Working with this definition, we prove the following result which is the exact analogue of Theorem \ref{thm-BP} for Weyl multipliers.

\begin{theorem} \label{Main-Weyl}
Let $M$ be a non-trivial (left) Weyl multiplier on $W^{N,1}_{L}(\mathbb{C}^{n})$ for some positive integer $N$. If $M$ is homogeneous of degree zero then $M$ has to be a constant multiple of the identity operator. 
\end{theorem}

\medskip \noindent \textbf{Organisation of the paper:}  In  Section \ref{sec-prelim} we recall  all the relevant preliminaries on the Heisenberg group and the Weyl transform and also define the associated Sobolev spaces. In subsection \ref{subsec-sph-harmonics} we describe Geller's analysis \cite{Geller} on operator analogues of spherical harmonics, and then define a notion of zero homogeneity for operators  in subsection \ref{subsec-hom-degree-0}. The proof of Theorem \ref{Main-Weyl} for Weyl multipliers on Laguerre Sobolev spaces is developed in Section \ref{sec-Weyl-mult-proof}. We discuss Fourier multipliers on the Heisenberg group in Section \ref{sec-mult-Heisenberg-proof}, finally proving Theorem \ref{thm2} in subsection \ref{subsec-hom-mult-Sob}.  

\medskip \noindent \textbf{Notations:} We denote by $\mathbb{N}$ the set of all non-negative integers $\{0, 1, 2, 3, \ldots\}$. For $z = (z_1, \ldots, z_n) \in \mathbb{C}^n$, we denote by $|z|$ the Euclidean norm $\left(|z_1|^2 + \ldots |z_n|^2\right)^{1/2}$, whereas, for a multi-index $\alpha = (\alpha_1, \ldots, \alpha_n) \in \mathbb{N}^n$, we use the same symbol $|\cdot|$ to denote the $l^1$ sum, that is, $|\alpha| = \alpha_1 + \ldots \alpha_n$. Also, $z^\alpha$ stands for $\prod_{j=1}^n z_j^{\alpha_j}$. Throughout the article, $W^{0,p}(\mathbb{R}^{2n})$ and $\mathring{W}^{0,p}(\mathbb{R}^{2n})$ should be understood as $L^p(\mathbb{R}^{2n})$. A similar remark applies to Laguerre Sobolev spaces and Sobolev spaces on the Heisenberg group. We use symbols such as $C$, $C_\epsilon$ etc to denote explicit/implicit positive constants, with $C_\epsilon$ may be depending on the parameter $\epsilon$. These constants need not be equal at different occurrences unless explicitly mentioned. For any two non-negative numbers $A$ and $B$, we write $A \lesssim_\epsilon B$ when there exists some $C_\epsilon >0$ such that $A \leq C_\epsilon B$.

%%%%%%%%%%%%%%%%%%%%%%%%%%%%%%%%%%%%%%%%%%%%%%%%%%%%%%%%%%%%%%%%%%%%%%%%%
%%%%%%%%%%%%%%%%%%%%%%%%%%%%%%%%%%%%%%%%%%%%%%%%%%%%%%%%%%%%%%%%%%%%%%%%%
%%%%%%%%%%%%%%%%%%%%%%%%%%%%%%%%%%%%%%%%%%%%%%%%%%%%%%%%%%%%%%%%%%%%%%%%% %%%%%%%%%%%%%%%%%%%%%%%%%%%%%%%%%%%%%%%%%%%%%%%%%%%%%%%%%%%%%%%%%%%%%%%%%
%%%%%%%%%%%%%%%%%%%%%%%%%%%%%%%%%%%%%%%%%%%%%%%%%%%%%%%%%%%%%%%%%%%%%%%%%
%%%%%%%%%%%%%%%%%%%%%%%%%%%%%%%%%%%%%%%%%%%%%%%%%%%%%%%%%%%%%%%%%%%%%%%%%
%%%%%%%%%%%%%%%%%%%%%%%%%%%%%%%%%%%%%%%%%%%%%%%%%%%%%%%%%%%%%%%%%%%%%%%%%
%%%%%%%%%%%%%%%%%%%%%%%%%%%%%%%%%%%%%%%%%%%%%%%%%%%%%%%%%%%%%%%%%%%%%%%%%
%%%%%%%%%%%%%%%%%%%%%%%%%%%%%%%%%%%%%%%%%%%%%%%%%%%%%%%%%%%%%%%%%%%%%%%%%

\section{Preliminaries and basic results} \label{sec-prelim}

In this section we set up the notation by defining the Heisenberg group and recalling all the relevant definitions and basic results that are required for this work.  Apart from Fourier and Weyl transforms we also need to recall the notion of Weyl correspondence in order to describe Geller's work on operator analogues of spherical harmonics. We also  introduce the sublaplacian $ \mathcal{L} $ and the special Hermite operators $ L(\lambda) $ and define the associated Sobolev spaces. The main references for this section are the monographs \cite{Folland},\cite{thangavelu} and the paper \cite{Geller}.

\subsection{Heisenberg group, Weyl transform and Sobolev spaces} \label{subsec-Heisen}
Let $\mathbb{H}^{n}$ denote the $(2n+1)$ dimensional Heisenberg group with the group law
$$(z,t)(w,s)=(z+w, t+s+ \frac{1}{2} \Im (z \cdot \bar{w})).$$ 
The Haar measure on $\mathbb{H}^{n}$ is the Lebesgue measure $dz \, dt$ of $\mathbb{C}^{n} \times \mathbb{R}$. For each fixed $\lambda \in  \R^\ast= \mathbb{R} \setminus \left\lbrace 0\right\rbrace$, we have an irreducible unitary representation $\pi_{\lambda}$ of $\mathbb{H}^{n}$ realised on $L^{2}(\mathbb{R}^{n})$, defined by 
$$\pi_{\lambda}(z,t) \phi(\xi) = e^{i \lambda t} e^{i \lambda (x\cdot\xi+ \frac{1}{2}x\cdot y)} \phi(\xi+y),$$ 
where $\phi\in L^{2}(\mathbb{R}^{n})$ and $z=x+ iy$. By the celebrated theorem of Stone-von Neumann, up to unitary equivalence these are all the irreducible unitary representations of $\mathbb{H}^{n}$ which are nontrivial at centre (see \cite{Folland}). These representations are used in defining the Fourier transform on the Heisenberg group.\\

Given $f \in L^1(\mathbb{H}^{n})$ and $ \lambda \in \R^\ast $ we can integrate $ f $ against $ \pi_\lambda $ to get a bounded linear operator
$$\widehat{f}(\lambda) = \int_{\mathbb{H}^{n}} f(z,t) \pi_{\lambda}(z,t) \, dz \, dt. $$ 
The operator valued function $\lambda \to \widehat{f}(\lambda)$ is called the (group) Fourier transform of $f$ on $\mathbb{H}^{n}$. For $f, g \in L^1(\mathbb{H}^{n}),$ we define the  convolution $ f \ast g $ by
$$f \ast g (z,t)=\int_{\mathbb{H}^{n}} f\left((z,t)(w,s)^{-1}\right) g(w,s) \, dw \, ds.$$
It then follows by direct verification that $ \widehat{f \ast g}(\lambda) = \widehat{f}(\lambda) \widehat{g}(\lambda).$ For $ f \in L^1 \cap L^2(\mathbb{H}^n) $ it can be proved that $ \widehat{f}(\lambda) $ is a Hilbert-Schmidt operator and we have the Plancherel theorem
\begin{equation} \label{planch} 
\int_{\mathbb{H}^n} |f(z,t)|^2 \, dz\, dt = (2\pi)^{-n-1} \int_{-\infty}^\infty \| \widehat{f}(\lambda)\|_{HS}^2 |\lambda|^n \, d\lambda. 
\end{equation}
If we denote by $f^\lambda$ the inverse Fourier transform of $f$ in the last variable at the point $\lambda,$ that is, $f^\lambda(z) = \int_{\mathbb{R}} f(z,t) e^{i \lambda t} \, dt,$ then we can easily verify that $(f \ast g)^{\lambda} = f^{\lambda} \ast_{\lambda} g^{\lambda},$ where $\ast_{\lambda}$ is called the $\lambda-$twisted convolution on $\mathbb{C}^n$ and is defined by 
$$F \ast_{\lambda} G (z) = \int_{\mathbb{C}^{n}}F(z-w)G(w) e^{i \frac{\lambda}{2}  \Im(z\cdot\overline{w})} \, dw,$$
for any $F, G \in L^1(\mathbb{C}^n).$ When $ \lambda = 1 $ we  write $ F \times G $ instead of $ F \ast_1 G $ and call this the twisted convolution of $ F $ with $ G.$\\

The above relation motivates one to define the Weyl transform on $\mathbb{C}^n$. More precisely, the Weyl transform $W_{\lambda}(f)$ of $f\in L^{1}\cap L^{2}(\mathbb{C}^{n})$ is an operator on $L^{2}(\mathbb{R}^{n})$ defined by 
\begin{align} \label{def-Weyl}
W_{\lambda}(f)\phi(\xi) := \int_{\mathbb{C}^{n}}f(z)\pi_{\lambda} (z,0)\phi(\xi) \, dz = \int_{\mathbb{C}^{n}} f(z) e^{i \lambda (x\cdot\xi+ \frac{1}{2} x \cdot y)} \phi(\xi+y) \, dz, 
\end{align}
where $z = x+iy$ and $\phi \in L^{2}(\mathbb{R}^{n})$. It is known that $W_{\lambda}$ maps $L^{1}(\mathbb{C}^{n})$ to the space of bounded operators on $L^{2}(\mathbb{R}^{n})$ and $W_{\lambda}$ maps $L^{2}(\mathbb{C}^{n})$ unitarily onto the space of Hilbert-Schmidt operators on $L^{2}(\mathbb{R}^{n})$. We also have the Plancherel theorem:
$$  \int_{\C^n} |f(z)|^2 dz  = C_n |\lambda|^n \| W_\lambda(f)\|_{HS}^2.$$ Weyl transform also satisfies the following property, namely, it takes twisted convolution to the composition of operators:  
$$W_{\lambda}(f \ast_{\lambda}g) = W_{\lambda}(f) W_{\lambda}(g).$$ 
When $ \lambda = 1 $ we simply write $ W(f) $ instead of $ W_1(f).$\\

Along with Weyl transform we also require the closely related notion of Weyl correspondence. The symplectic Fourier transform of a function $ f \in L^1(\C^n) $ is defined by
$$  \mathcal{F}_\lambda f(z) = (2\pi)^{-n} \int_{\C^n} f(z-w) e^{i \frac{\lambda}{2} \Im(z \cdot \bar{w})} \, dw  = f \ast_\lambda 1(z).$$ Note that the symplectic Fourier transform is related to the ordinary Fourier transform: $ \mathcal{F}_\lambda f(z) = \widehat{f}(-\frac{i}{2}\lambda z).$ The symplectic Fourier transform has  natural extension to all tempered distributions. We define the Weyl correspondence of a tempered distribution $ f $ by $ G_\lambda(f) = W_\lambda(\mathcal{F}_\lambda f).$ A priori it is not clear if $ G_\lambda(f) $ defines an operator on $ L^2(\R^n).$ When $ f = P $ is a polynomial, $ \mathcal{F}_\lambda P $ is a finite sum of derivatives of the Dirac delta and hence $ G_\lambda(P) $ turns out to be a differential operator. We will make use of this fact in defining operator analogues of spherical harmonics.\\

The Heisenberg Lie algebra $ \mathfrak{h}^n $ is spanned by the following $ (2n+1) $ left invariant vector fields
$$ X_j =  \frac{\partial}{\partial x_j} + \frac{1}{2} y_j \frac{\partial}{\partial t}, \quad Y_j =  \frac{\partial}{\partial y_j} - \frac{1}{2} x_j \frac{\partial}{\partial t}, \quad j = 1, 2, \ldots, n, $$
and $ T = \frac{\partial}{\partial t}.$ The operator $ \mathcal{L} = - \sum_{j=1}^n \left(  X_j^2+Y_j^2 \right)  ,$ known as the sublaplacian plays the role of $ -\Delta $ for the Heisenberg group. More explicitly $ \mathcal{L} $ is given by
$$  \mathcal{L} = -\Delta_{\C^n} - \frac{1}{4} |z|^2 \frac{\partial^2}{\partial t^2} +  \sum_{j=1}^n \left(  x_j \frac{\partial}{\partial y_j} - y_j \frac{\partial}{\partial x_j} \right)  \frac{\partial}{\partial t} .$$
Along with the left invariant vector fields, we also need their right invariant analogues which are given by
$$ \tilde{X}_j =  \frac{\partial}{\partial x_j} - \frac{1}{2} y_j \frac{\partial}{\partial t}, \quad \tilde{Y}_j =  \frac{\partial}{\partial y_j} +\frac{1}{2} x_j \frac{\partial}{\partial t}, \quad j = 1, 2, \ldots, n,$$
and the right invariant sublaplacian is defined by $ \tilde{\mathcal{L}} = - \sum_{j=1}^n \left(  \tilde{X}_j^2+\tilde{Y}_j^2 \right)  .$  These operators generate contraction semigroups $ e^{-t\mathcal{L}} $ and $ e^{-t\tilde{\mathcal{L}}} $ which are given by convolutions with an explicit kernel:
$$ e^{-t\mathcal{L}}f(g) = f \ast p_t(g), \quad  e^{-t\tilde{\mathcal{L}}}f(g) = p_t \ast f(g), $$
where 
\begin{equation} \label{heat} 
p_t(w,s) =  2^{-n} (2\pi)^{-n-1}  \int_{-\infty}^\infty e^{-i\lambda s}\left( \frac{\lambda}{\sinh(t\lambda)}\right) ^n e^{-\frac{1}{4}\lambda (\coth(t\lambda))|w|^2} d\lambda. 
\end{equation}
Under the Fourier transform, the sublaplacians get converted into the Hermite operator $ H(\lambda) = -\Delta+\lambda^2 |x|^2$:
$$ \widehat{(\mathcal{L}f)}(\lambda) = \widehat{f}(\lambda) H(\lambda), \quad \widehat{(\tilde{\mathcal{L}}f)}(\lambda) = H(\lambda) \widehat{f}(\lambda). $$
These relations lead to the formulas
$$ \widehat{f\ast p_t}(\lambda) = \widehat{f}(\lambda) e^{-t H(\lambda)}, \quad \widehat{ p_t\ast f}(\lambda) = e^{-t H(\lambda)} \widehat{f}(\lambda).$$
For more about the heat kernel we refer to the monograph \cite{thangavelu}. \\

Using the vector fields $ X_j, Y_j, j=1, 2, \ldots, n $ we can define Sobolev spaces $ W^{N,p}(\He^n) $ and the homogeneous Sobolev spaces $ \mathring{W}^{N,p}(\He^n)$.  For multi-indices $ \alpha$ and $ \beta $ we set $ X^\alpha = \Pi_{j=1}^n X_j^{\alpha_j}, Y^\beta = \Pi_{j=1}^n Y_j^{\beta_j}.$ 
As in the Euclidean case we let
$$  W^{N,p}(\He^n) =  \{  f \in L^p(\He^n):  X^\alpha Y^\beta T^jf \in L^p(\He^n),  |\alpha|+|\beta|+2j \leq N\} $$
and we equip $W^{N,p}(\He^n) $ with the norm $ \| f\|_{W^{N,p}} =  \sum_{|\alpha|+|\beta|+2j \leq N} \| X^\alpha Y^\beta T^jf\|_p.$ Note that $X^\alpha Y^\beta T^j $ is homogeneous of degree $ |\alpha|+|\beta|+2j $ with respect to non-isotropic dilations. So it makes sense to define the homogeneous Sobolev spaces as
$$  \mathring{W}^{N,p}(\He^n) =  \{  f \in \mathcal{S}^\prime(\He^n):  X^\alpha Y^\beta T^jf \in L^p(\He^n),  |\alpha|+|\beta|+2j = N\} $$
equipped with the norm $ \| f\|_{\mathring{W}^{N,p}} =  \sum_{|\alpha|+|\beta|+2j = N} \| X^\alpha Y^\beta T^jf\|_p.$ 

Observe that when $ N =1 $ these Sobolev spaces are defined solely in terms of $ X_j $ and $ Y_j $ and there is no requirement on $ T.$\\

The above spaces are left-invariant in the sense that for any $ g \in \He^n $ the function $ \tau_g f $ defined by $ \tau_gf(h) = f(g^{-1}h) $ belongs to $W^{N,p}(\He^n)$ whenever $ f \in W^{N,p}(\He^n).$ And a similar remark applies to the homogeneous Sobolev spaces. Replacing the left-invariant vector fields by the right-invariant ones, viz. $\tilde{X_j}, \tilde{Y_j} $ we can define $ W_R^{N,p}(\He^n) $ and $ \mathring{W}_R^{N,p}(\He^n).$ These spaces are invariant under right translations. We remark that the map $ f \rightarrow f^\ast $ defined by $ f^\ast(g) = \overline{f(g^{-1})} $ takes $ W^{N,p}(\He^n)$ isometrically onto $W_R^{N,p}(\He^n) .$ This is a consequence of the easily verifiable fact that $ (X_jf)^\ast = \tilde{X_j}f^\ast, (Y_jf)^\ast = \tilde{Y_j}f^\ast.$  A similar remark applies to homogeneous Sobolev spaces also. We will make use of these relations in the sequel. 

%%%%%%%%%%%%%%%%%%%%%%%%%%%%%%%%%%%%%%%%%%%%%%%%%%%%%%%%%%%%%%%%%%%%%%%%%
%%%%%%%%%%%%%%%%%%%%%%%%%%%%%%%%%%%%%%%%%%%%%%%%%%%%%%%%%%%%%%%%%%%%%%%%%

\subsection{Fourier multipliers on the Heisenberg group} \label{subsec-Heisen-mult} 
Given an $ L^\infty $ function $ m $ on $ \R^\ast $ taking values in $ B(L^2(\R^n)) $ we can define an operator $ T_m $ on $L^2(\He^n) $ by the prescription $ \widehat{T_mf}(\lambda) = m(\lambda) \widehat{f}(\lambda) .$ In view of  the Plancherel theorem (\ref{planch}) it is immediate that such an operator is bounded on $ L^2(\He^n). $ However, without further conditions on $ m $ it need not extend from $ L^p \cap L^2(\He^n) $ to $ L^p(\He^n) $ as a bounded operator. If it happens, we say that $ m $ is an $ L^p $ Fourier multiplier on the Heisenberg group. Equivalently, we also say that $ T_m $ is a left Fourier multiplier operator on $ L^p(\He^n).$ We can also define right Fourier multiplier operators by $ \widehat{\tilde{T}_mf}(\lambda) = \widehat{f}(\lambda) m(\lambda).$ It is easy to see that left (resp. right) Fourier multiplier operators are invariant under right (resp. left) translations on the Heisenberg group.  As in the case of Fourier multipliers on $ \R^n $ these operators can be realised as convolution operators.\\

Let $ \mathcal{S}_1 $ stand for the ideal of trace class operators in $ B(L^2(\R^n)) $ and consider the space $ L^1(\R^\ast, \mathcal{S}_1, d\mu) $ where $ d\mu(\lambda) = (2\pi)^{-n-1} |\lambda|^n \, d\lambda $ is the Plancherel measure for $ \He^n.$  To every $ \varphi \in L^1(\R^\ast, \mathcal{S}_1, d\mu) $ we can associate the function
$$  \widetilde{\varphi}(z,t) = \int_{-\infty}^\infty  tr(\pi_\lambda(z,t)^\ast \varphi(\lambda)) d\mu. $$
The Fourier algebra $ A(\He^n) $ is then defined to be the space of all such functions $ \widetilde{\varphi} $ where $ \varphi \in L^1(\R^\ast, \mathcal{S}_1, d\mu) $ which becomes a Banach space when equipped with the norm
$$ \| \widetilde{\varphi} \|_{A} = \int_{-\infty}^\infty  tr(|\varphi(\lambda)|) d\mu .$$  We denote by  $ P(\He^n) $  the dual of $ A(\He^n).$ The elements of $ P(\He^n) $  are called pseudo-measures and it is known that to every left  Fourier multiplier we can associate a pseudo-measure $S $ such that $ T_m f = S \ast f,$ see the works \cite{MC} and \cite{RV}. In the case of right Fourier multipliers we have $ \tilde{T}_m f = f \ast S.$ It has been proved in \cite{RV} that $ P(\He^n) $ can also be identified with the dual of $ L^1(\R^\ast, \mathcal{S}_1, d\mu) $ and consequently, it is possible to talk about the Fourier transform of pseudo-measures.\\

As we are interested in the boundedness of left Fourier multipliers, a natural question that arises is the following: which Sobolev space is the most suitable one to work with? Since the operator $ T_m $ is right invariant, it may look natural to study the boundedness on the right invariant Sobolev space $W_R^{N,p}(\He^n) .$ However, the left invariant spaces $ W^{N,p}(\He^n) $ seem to be more suitable for the following reason. Suppose we know that $ T_m $ is bounded on $ L^p(\He^n) $ then it is immediate that it is also bounded on $ W^{N,p}(\He^n) $ since $ X_j (T_mf) = T_m(X_jf), Y_j(T_mf) = T_m(Y_jf).$ Thus the class of left Fourier multipliers on $ L^p(\He^n) $ is included in the class of left Fourier multipliers on $ W^{N,p}(\He^n).$ We do not have such an inclusion in the case of $ W_R^{N,p}(\He^n).$
In fact the boundedness of $ T_m $ on $ L^p(\He^n) $ and on $ W_R^{1,p}(\He^n)$ implies that the commutators $ [\tilde{X_j},T_m], [\tilde{Y}_j,T_m] $ satisfy the estimates
$$ \| [\tilde{X_j},T_m]f\|_p + \|[\tilde{Y}_j,T_m]f\|_p \leq C \|f\|_{W_R^{1,p}} .$$ 
However, even for a simple minded operator such as $ T_m f = \nu \ast f $ where $ \nu $ is a finite Borel measure, it is not clear if such an estimate is true or not. This explains why in Theorem \ref{thm2} we have considered $ T_m $ on $ W^{N,1}(\He^n).$

%%%%%%%%%%%%%%%%%%%%%%%%%%%%%%%%%%%%%%%%%%%%%%%%%%%%%%%%%%%%%%%%%%%%%%%%%
%%%%%%%%%%%%%%%%%%%%%%%%%%%%%%%%%%%%%%%%%%%%%%%%%%%%%%%%%%%%%%%%%%%%%%%%%

\subsection{Homogeneous multipliers on the Heisenberg group} \label{subsec-Heisen-hom-mult}
In the case of  Fourier multipliers  on $ \R^n $ we note that, with $ d_rf $ standing for the dilation $ d_rf(x) = f(rx), r>0, $  
\begin{equation}\label{dil-R} T_m (d_rf)(x) = r^{-n} (2\pi)^{-n} \int_{\R^n} e^{i x \cdot \xi} m(\xi) \widehat{f}(r^{-1}\xi) d\xi
\end{equation}
or equivalently, 
we have  $ T_m(d_rf)(r^{-1}x) =  T_{d_rm}f(x).$ Thus we see that the multiplier $ m $ is homogeneous of degree zero if and only $ T_m $ commutes with the dilation $ d_r $ for every $ r >0.$
The dilations $ d_r $ are automorphisms of the group $ \R^n $ and their counter parts in the context of the Heisenberg group are given by the non-isotropic dilations $ \delta_r $ defined by $ \delta_r(z,t) = (rz,r^2t).$ The connection between the (group) Fourier transforms of $ f $ and $ \delta_rf(z,t) = f(\delta_r(z,t)) $ for a function $ f $ on $ \He^n $ is given by
$$ \widehat{(\delta_r f)}(\lambda) = r^{-2(n+1)} d_r \circ \widehat{f}\left( \lambda / r^2 \right) \circ d_r^{-1} $$
for every $ \lambda \in \R^* $ and $ r >0. $  In view of this the analogue of (\ref{dil-R}) takes the form
$$ \widehat{(T_m\delta_rf)}( \lambda) = r^{-2(n+1)}  m(\lambda) d_r  \circ \widehat{f}\left( \lambda / r^2 \right) \circ d_r^{-1}.$$
From this we see that $ T_m $ commutes with $ \delta_r $ if and only if we have the relation
$$ d_r \circ m\left( \lambda / r^2 \right) \circ d_r^{-1} = m(\lambda).$$
Thus we see that $ m(\lambda) $ is completely determined by $ m(1)$ and $ m(-1).$  Indeed, 
$$  m(\lambda)  = d_{\sqrt{|\lambda|}} \circ  m\left( \lambda / |\lambda| \right) \circ d_{\sqrt{|\lambda|}}^{-1} .$$
We may say that Fourier multiplier $ m $ is homogeneous of degree zero if it satisfies the above relation.\\

The most important Fourier multipliers which are homogeneous of degree zero are provided by the Riesz transforms $ \tilde{R}_j =  \tilde{X}_j  \tilde{\mathcal{L}}^{-1/2},\, \tilde{R}_{j+n} =  \tilde{Y}_j \tilde{\mathcal{L}}^{-1/2}$ and their higher order analogues. The associated multipliers are given by $ i\lambda \xi_j H(\lambda)^{-1/2} $ and $ \frac{\partial}{\partial \xi_j} H(\lambda)^{-1/2} $ where $ H(\lambda) = -\Delta+\lambda^2 |\xi|^2 ,$ see \cite{thangavelu}. In this definition, the fractional powers $\tilde{\mathcal{L}}^{-1/2} $ are defined in terms of the heat kernel, and it is not difficult to check that they are homogeneous of degree $-1.$ Since $\tilde{X}_j $ are homogeneous of degree one, it follows that $  \tilde{R}_j $ commutes with the dilations. We can also verify this at the level of multipliers. Indeed, if $ P_k(\lambda) $ are the projections associated to $ H(\lambda),$ then we have
$$ P_k(\lambda)  = d_{\sqrt{|\lambda|}} \circ  P_k \circ d_{\sqrt{|\lambda|}}^{-1}, \quad H(\lambda)^{-1/2}  = |\lambda|^{-1/2} d_{\sqrt{|\lambda|}} \circ H^{-1/2} \circ d_{\sqrt{|\lambda|}}^{-1}  $$
where $ P_k = P_k(1) $ and similar relations (with $ |\lambda|^{-1/2} $ replaced with $ \lambda^{1/2} $) hold for the operators $ i\lambda \xi_j $ and $ \frac{\partial}{\partial \xi_j}$ proving our claim.
It is easy to produce other examples of multipliers that are homogeneous of degree zero. For any bounded sequence  $ c_k $  the operators 
$$  m(\lambda)  = \sum_{k=0}^\infty  c_k P_k(\lambda) $$ are uniformly bounded, homogeneous of degree zero and define Fourier multiplier operator $ T_m $ which are bounded on $ L^2(\He^n).$\\

Our definition of homogeneity has the following consequence on the pseudo-measure $ S $ associated to $ T_m.$ For any $ r > 0 ,$ we define $ \delta_r S $ by the relation
$ \langle \delta_r S,\tilde{\varphi}\rangle =  r^{-Q} \langle S, \delta_r^{-1}\tilde{\varphi}\rangle $ for any $\tilde{\varphi} \in A(\He^n).$ Then $ T_m $ commutes with $ \delta_r $ if and only if  $ S $ is homogeneous of degree $ -Q $ in the sense that $ \delta_rS = r^{-Q} S$ which reduces to $ \langle S, \tilde{\varphi} \rangle = \langle S, \delta_r^{-1}\tilde{\varphi} \rangle .$

%%%%%%%%%%%%%%%%%%%%%%%%%%%%%%%%%%%%%%%%%%%%%%%%%%%%%%%%
%%%%%%%%%%%%%%%%%%%%%%%%%%%%%%%%%%%%%%%%%%%%%%%%%%%%%%%%

\subsection{The special Hermite operator and Laguerre-Sobolev spaces} \label{subsec-Laguerre-Sob} 
The action of the sublaplacian $ \mathcal{L} $ on functions of the form $ f(z,t) = e^{i\lambda t} F(z) $ gives rise  to the following family of operators $L(\lambda)$ defined by $ \mathcal{L}f(z,t) = e^{i\lambda t} L(\lambda) F(z).$ More explicitly,
$$L(\lambda) = - \Delta_{\C^n} + \frac{1}{4} \lambda^2 |z|^2 + i \lambda \sum_{j=1}^n\left(x_j \frac{\partial}{\partial y_j} - y_j \frac{\partial}{\partial x_j}\right).$$
These operators are called special Hermite operators and their spectral theory is known explicitly, see \cite{thangavelu}. By introducing the complex vector fields $ \mathcal{Z}_j$ and $ \overline{\mathcal{Z}}_j$ defined by  $ \mathcal{Z}_j  = \frac{1}{2} ( X_j - i Y_j)  $ and $ \overline{\mathcal{Z}}_j = \frac{1}{2} ( X_j + i Y_j) $ we get the vector fields $ Z_j(\lambda) $ and $ \overline{Z}_j (\lambda) $ defined by 
$ \mathcal{Z}_j f(z,t) = e^{i\lambda t}  Z_j(\lambda) F(z)$ and $ \overline{\mathcal{Z}}_j f(z,t) = e^{i\lambda t} \overline{Z}_j (\lambda)F(z)$ and we see that 
\begin{equation} \label{TL}
L(\lambda) = - 4 \sum_{j=1}^{n} \left(Z_{j}(\lambda)\overline{Z}_{j}(\lambda) + \overline{Z}_{j}(\lambda)Z_{j}(\lambda)\right).
\end{equation}
We also observe that 
the differential operators $Z_{j}(\lambda)$ and  $\overline{Z}_{j}(\lambda)$, $j = 1, \ldots, n$ are given by 
\begin{equation} Z_{j}(\lambda)=\frac{\partial}{\partial z_{j}} - \frac{\lambda}{4} \overline{z}_{j}, \quad  \overline{Z}_{j}(\lambda) = \frac{\partial}{\partial \overline{z}_{j}} + \frac{\lambda}{4} z_{j}
\end{equation}
where $ \frac{\partial}{\partial z_j} $ and $ \frac{\partial}{\partial \bar{z}_j} $ have the usual meaning. Along with these we also need their right invariant counter parts
\begin{align} \label{scaled-right-fields}
Z_j^R(\lambda) =\frac{\partial}{\partial z_{j}} + \frac{\lambda}{4} \overline{z}_{j}, \quad {\overline{Z}_j^R} (\lambda) = \frac{\partial}{\partial \overline{z}_{j}} - \frac{\lambda}{4}z_{j}, \quad j = 1, \ldots, n. 
\end{align}
We observe that $ Z_j^R(\lambda) = Z_j(-\lambda) $ and $ \overline{Z}_j^R(\lambda) = \overline{Z}_j(-\lambda).$
The operators  $Z_{j}(\lambda), \overline{Z}_{j}(\lambda)$ are connected to $A_j(\lambda)= \frac{\partial}{\partial x_j}+\lambda x_j $ and $ {A}_j(\lambda)^* = - \frac{\partial}{\partial x_j}+\lambda x_j $ via the Weyl transform in the sense that for $ \lambda > 0$
$$ W_{\lambda}(Z_{j}(\lambda)f) = \frac{i}{2} W_{\lambda}(f)A_{j}(\lambda)^*, \quad W_{\lambda}(\overline{Z}_{j}(\lambda)f) = \frac{i}{2} W_{\lambda}(f){A}_{j}(\lambda),$$
and therefore,
\begin{equation}
W_{\lambda}(L(\lambda)f) = 4 W_{\lambda}(f)H(\lambda). 
\end{equation}
For $ \lambda > 0 $ we also have  
$$W_{\lambda}(Z_{j}(-\lambda)f) = \frac{i}{2} A_{j}(\lambda)^* W_\lambda(f), \quad W_{\lambda}(\overline{Z}_{j}(-\lambda)f) = \frac{i}{2} {A}_{j}(\lambda) W_\lambda(f).$$When $ \lambda = 1 $ we suppress the parameter and simply write $ W, L $ and $ H$ instead of $ W(1), L(1) $ and $ H(1).$ The same convention will be followed with other notations as well. \\

The operator $ L(\lambda) $ has discrete spectrum and has a very explicit spectral decomposition. For all the results stated in this section without proof we refer to the monograph \cite{thangavelu}. We let
$$ \varphi_k^{n-1}(z) = L_k^{n-1} \left( \frac{1}{2} |z|^2 \right) e^{-\frac{1}{4} |z|^2} $$ stand for Laguerre functions of type $(n-1).$ Here $ L_k^\alpha(t), \alpha > -1 $ are Laguerre polynomials of type $ \alpha.$ For each $ \lambda \in \R^\ast $ we define $ \varphi_{k,\lambda}^{n-1}(z) = \varphi_{k}^{n-1}(\sqrt{|\lambda|}z).$ 
Then for any $ f \in L^2(\C^n) $ the twisted convolution $ (2\pi)^{-n} |\lambda|^n f \ast_\lambda \varphi_{k,\lambda}^{n-1} $ gives the spectral projection of $ L(\lambda) $ onto the eigenspace corresponding to the eigenvalue $ (2k+n)|\lambda| $ and we have
$$ f(z)  = (2\pi)^{-n} |\lambda|^n \sum_{k=0}^\infty  f \ast_\lambda \varphi_{k,\lambda}^{n-1}(z).$$
The heat kernel associated to $ L(\lambda) $ is explicitly known and given by
$$ p_t^\lambda(z) = (4\pi)^{-n} \left( \frac{\lambda}{\sinh(t\lambda)}\right) ^n  e^{-\frac{\lambda}{4}  (\coth(t\lambda)) |z|^2}.$$ 
The function $ u(z,t) = f \ast_\lambda p_t^\lambda(z) $ solves the heat equation associated to $ L(\lambda) $ with initial condition $ f.$ The fractional powers $ L(\lambda)^{-s}, \, s > 0 $ are expressible in terms of  the heat semigroup $ e^{-tL(\lambda)}.$ Thus,
$$ L(\lambda)^{-s} f(z) = \frac{1}{\Gamma(s)} \int_0^\infty  t^{s-1} e^{-tL(\lambda)}f(z) \, dt.$$
As $ e^{-tL(\lambda)}f(z) = f \ast_\lambda p_t^\lambda(z) $ it follows that $ L(\lambda)^{-s}f(z) = f \ast_\lambda K^s_{\lambda} (z) $ where
$$ K^s_{\lambda}(z) = \frac{1}{\Gamma(s)} \int_0^\infty  t^{s-1} p_t^\lambda(z) \, dt.$$ 
More generally, the kernel $ K^s_{\lambda,d} $ of $ (L(\lambda) + d |\lambda|)^{-s} $ is given by
$$ K^s_{\lambda,d}(z) = \frac{1}{\Gamma(s)} \int_0^\infty  t^{s-1} e^{-d |\lambda|t} p_t^\lambda(z) \, dt.$$ 
We record the following result, which is easy to prove, for future use.

\begin{proposition} \label{prop-inverse-frac-twisted-lap}
The kernel $ K^s_{\lambda,d} $ is well defined for all $ s > 0$ and $ d+n > 0$ and belongs to $ L^1(\C^n).$ Consequently, $ (L(\lambda) + d|\lambda|)^{-s} $ is bounded on $ L^p(\C^n) $ for all $ 1 \leq p \leq \infty.$
\end{proposition}
\begin{proof} The proof follows from the fact that
$$ \int_{\C^n} p_t^\lambda(z) \, dz = C_n (\cosh(t|\lambda|)^{-n} \leq C_n e^{- n |\lambda| t} $$
so that $ \int_0^\infty  t^{s-1} e^{-(d+n)|\lambda| t} \, dt <\infty$ under the assumptions on $ s $ and $ d.$
\end{proof}

We now define the Laguerre- Sobolev spaces $ W_{L(\lambda)}^{N,p}(\C^n) $ for any non-negative integer $ N $ and $ 1 \leq p \leq \infty.$ (By our convention $ W_L^{N,p}(\C^n) $ will stand for $ W_{L(1)}^{N,p}(\C^n) .$)

\begin{definition}[Laguerre Sobolev Spaces] \label{Lag-Sob-Sp}
Let $1 \leq p \leq \infty$, $N \in \mathbb{N} \setminus \{0\}$ and $\lambda \in \R^\ast$. The Laguerre Sobolev space $W^{N,p}_{L(\lambda)}(\mathbb{C}^{n})$ is defined by 
\begin{align*} %\label{LS-space}
W^{N,p}_{L(\lambda)}(\mathbb{C}^{n}) = \left\lbrace f \in L^p(\mathbb{C}^{n}) : Z(\lambda)^\alpha \overline{Z}(\lambda)^\beta f \in L^p(\mathbb{C}^{n}), 0 \leq |\alpha| + |\beta| \leq N\right\rbrace,
\end{align*}
where for multi-indices $ \alpha, \beta \in \N^n, $ we have set $ \,Z(\lambda)^\alpha = \Pi_{j=1}^n Z_j(\lambda)^{\alpha_j} $ and $\overline{Z}(\lambda)^\beta  = \Pi_{j=1}^n\overline{Z}_j(\lambda)^{\beta_j}.$ 
\end{definition} 

It is easy to see  that  $W^{N,p}_{L(\lambda)}(\mathbb{C}^{n})$ is a Banach space with respect to the norm 
\begin{align*} %\label{LS-norm}
\left\| f\right\|_{W^{N,p}_{L(\lambda)}} := \sum_{|\alpha| + |\beta| \leq N} \left\| Z(\lambda)^\alpha \overline{Z}(\lambda)^\beta f\right\|_p.
\end{align*}
Similarly, one defines the homogeneous Laguerre Sobolev space $\mathring{W}^{N,p}_{L(\lambda)}(\mathbb{C}^{n})$ to be the space of all tempered distributions $f$ such that $Z(\lambda)^\alpha \overline{Z}(\lambda)^\beta f \in L^p(\mathbb{C}^{n})$, for $|\alpha| + |\beta| = N$, equipped with the seminorm 
\begin{align*} %\label{hom-LS-norm}
\left\| f\right\|_{\mathring{W}^{N,p}_{L(\lambda)}} := \sum_{|\alpha| + |\beta| = N} \left\| Z(\lambda)^\alpha \overline{Z}(\lambda)^\beta f\right\|_p. 
\end{align*}
It turns out that for $1 \leq p < \infty$, the two spaces ${W}^{N,p}_{L(\lambda)}(\mathbb{C}^{n})$ and $\mathring{W}^{N,p}_{L(\lambda)}(\mathbb{C}^{n})$ coincide with the norm $\left\| \cdot\right\|_{W^{N,p}_{L(\lambda)}(\mathbb{C}^{n})}$ being equivalent to $\left\| \cdot\right\|_{\mathring{W}^{N,p}_{L(\lambda)}(\mathbb{C}^{n})}$.  This is due to the fact that the spectrum of $ L(\lambda) $ is discrete and does not contain $ 0.$ 

\begin{theorem} \label{thm-Weyl-norm-equiv}
Let $ \lambda \in \R^* $ and $ N \in \N.$ Then for any $ 1 \leq p < \infty$ we have ${W}^{N,p}_{L(\lambda)}(\mathbb{C}^{n})$ and $\mathring{W}^{N,p}_{L(\lambda)}(\mathbb{C}^{n})$ coincide with the norm $\left\| \cdot\right\|_{W^{N,p}_{L(\lambda)}(\mathbb{C}^{n})}$ being equivalent to $\left\| \cdot\right\|_{\mathring{W}^{N,p}_{L(\lambda)}(\mathbb{C}^{n})}$.
\end{theorem} 

For $ 1 < p < \infty $ the result follows from Proposition \ref{prop-inverse-frac-twisted-lap} together with the fact that the Riesz transforms $ Z_j(\lambda)L(\lambda)^{-1/2}$ and  $\overline{Z}_j (\lambda)L(\lambda)^{-1/2} $ and their higher order analogues are bounded on $ L^p(\C^n).$ For $ p =1 $ it follows from the Poincare type inequality stated in the following proposition.

\begin{proposition} For any $ \lambda \in \R^\ast $ we have the inequality
$$  \| f\|_1 \leq C_\lambda  \sum_{j=1}^n \left(  \| Z_j(\lambda)f\|_1+ \| \overline{Z}_j (\lambda)f\|_1 \right)  $$
for all $ f \in \mathring{W}^{1,1}_{L(\lambda)}(\mathbb{C}^{n})$. 
\end{proposition}
\begin{proof}  In view of (\ref{TL}), writing $ f = L(\lambda)^{-1} L(\lambda)f $, it is enough to prove that $ L(\lambda)^{-1} Z_j(\lambda) $ and $ L(\lambda)^{-1} \overline{Z}_j (\lambda) $ are bounded on $ L^1(\C^n).$  From the definition, we can check that $ L(\lambda) $ commutes with $ Z_j(-\lambda) $ and $ \overline{Z}_j (-\lambda) $ for all $ j .$  Let us define $ M_j f(z) = z_j f(z) $ and $ \overline{M}_j f(z) = \overline{z}_j f(z).$ As $ Z_j(\lambda) = Z_j(-\lambda)-\frac{\lambda}{2}\overline{z}_j, $ using the fact that $ L(\lambda) ^{-1} $ also commutes with $ Z_j(-\lambda)$ we have
\begin{equation}\label{comm}
 L(\lambda)^{-1} Z_j(\lambda) = Z_j(-\lambda) L(\lambda)^{-1} -\frac{\lambda}{2} L(\lambda)^{-1} \overline{M}_j =  Z_j(\lambda) L(\lambda)^{-1} - \frac{\lambda}{2} [L(\lambda)^{-1},
 \overline{M}_j] 
 \end{equation}
where  for any two operators $ [T,S] = TS-ST $ stands for their commutator. Expressing the operator $ L(\lambda) $ in terms of the semigroup $ e^{-tL(\lambda)} $ we see that
$$ L(\lambda)^{-1} f (z) = \int_0^\infty  f \ast_\lambda p_t^\lambda(z) \, dt $$ where $ p_t^\lambda(z) $ is the kernel of $ e^{-tL(\lambda)}.$ The kernel of $ L(\lambda)^{-1} $ is given by 
\begin{equation} 
\label{fund} K_\lambda(z) = \int_0^\infty p_t^\lambda(z) \, dt 
\end{equation} 
which can be calculated explicitly. The above means that
$$  L(\lambda)^{-1} f(z) =  \int_{\C^n} f(z-w) K_\lambda(w) e^{\frac{i}{2} \lambda \Im(z \cdot w)} dw.$$ Observe that the kernel of $Z_j(\lambda) L(\lambda)^{-1}$ is given by $Z_j(\lambda) K_\lambda(z) $ and that of $ [L(\lambda)^{-1},\overline{M}_j] $ is given by $ \overline{z}_j K_\lambda(z).$ Once we show that these kernels are in $ L^1(\C^n) $ it follows that $ L(\lambda)^{-1}Z_j(\lambda) $ are bounded on $ L^1(\C^n).$  Since similar analysis works for $ L(\lambda)^{-1} \overline{Z}_j (\lambda) $ as well the proposition gets proved.

Though the kernel $ K_\lambda(z) $ can be evaluated in terms of a Macdonald function, for our purpose it is enough to use the following integral representation
$$ K_\lambda(z)= c_{n,\lambda}  \, e^{-\frac{\lambda}{4}|z|^2}  \int_0^\infty  (s(s+2))^{n/2-1} e^{-\frac{\lambda}{4} s|z|^2} \, ds$$ 
which can be obtained from (\ref{fund}) by clever change of variables, see Theorem 3.1 in Adimurthi et al \cite{ARS}. From the above representation it is not difficult to get the asymptotics of $ K_\lambda(z) $ for $ |z| $ small as well as for $ |z|$ large. We refer to Propositions $4.3$ and $4.4$ of \cite{ARS}. Our claims follow immediately from the estimates proved there.
\end{proof}

%%%%%%%%%%%%%%%%%%%%%%%%%%%%%%%%%%%%%%%%%%%%%%%%%%%%%%%%%%%%%%%%%%%%%%%%%
%%%%%%%%%%%%%%%%%%%%%%%%%%%%%%%%%%%%%%%%%%%%%%%%%%%%%%%%%%%%%%%%%%%%%%%%%

\subsection{Spherical harmonics and their operator analogues} \label{subsec-sph-harmonics}
Let $K=U(n)$ be the group of all $n\times n$ unitary matrices which acts on any function space on the unit sphere $S^{2n-1}$. Let $K_{e_{1}}$ be the subgroup of $K$ that fixes the coordinate vector $e_{1}=(1,0,\ldots,0)$. Since $K$ acts transitively on $S^{2n-1}$, we can identify $S^{2n-1}$ with $K/K_{e_{1}}$ via the map $\omega \rightarrow u K_{e_{1}} $ if $ \omega = u\cdot e_{1}$. The natural representation of $K$ on $L^2(S^{2n-1})$ can be decomposed in terms of irreducible unitary representations having $K_{e_{1}}$-fixed vectors known as class one representations. For each pair $(a,b) \in \mathbb{N}^2$, let $\mathcal{P}_{a,b}$ be the set of all polynomials on $\mathbb{C}^n$ which are of the form 
$$P(z) = \sum_{|\alpha|=a}\sum_{|\beta|=b}a_{\alpha\beta}z^{\alpha}\overline{z}^{\beta}.$$ 
Each $P \in \mathcal{P}_{a,b}$ satisfies the homogeneity condition $P(\lambda z) = \lambda^{a} \bar{\lambda}^{b} P(z)$ for all $\lambda \in \mathbb{C} \setminus \{0\}$. Let $\Delta = 4 \sum_{j=1}^{n}\frac{\partial^{2}}{\partial z_{j}\partial\overline{z}_{j}}$ be the Laplacian on $\mathbb{C}^{n}$. We write $\mathcal{H}_{a,b} := \left\lbrace P\in\mathcal{P}_{a,b}: \Delta P=0 \right\rbrace$. The elements of $\mathcal{H}_{a,b}$ are called  bigraded solid harmonics. It is known that $\delta=\delta_{a,b}$ defined on $U(n)$ (and acting on $\mathcal{H}_{a,b}$) by $\delta(\sigma)P(z) = P(\sigma^{-1} z)$ are irreducible unitary representations and exhaust all class one irreducible unitary representations (upto unitary equivalence). We denote this class of representations by $\widehat{K_0}$. 
For each $\delta=\delta_{a,b},$ we let $d(\delta)$ denote the dimension of $\mathcal{H}_{a,b}$ and let $\chi_\delta$ be the character associated to $\delta$. \\

We make  $\mathcal{H}_{a,b}$  into a Hilbert space by equipping it with the inner product: 
$$(f,g)_{\mathcal{H}_{a,b}} = \frac{2^{-(n+a+b-1)}}{\Gamma(n+a+b)} \int_{\mathbb{C}^n} f(z) \overline{g(z)} e^{-\frac{1}{2} |z|^2} \, dz.$$ 
We fix an orthonormal basis $\{P^{\delta}_{j}$ : $1\leq j\leq d(\delta)\}$ for  $\mathcal{H}_{a,b}.$  Then by defining the spherical harmonics $Y_j^\delta$ by the relation $ P_j^\delta(z) = |z|^{a+b} Y_j^\delta(\omega),$  for $z= |z|\omega$, the collection
$$\left\lbrace Y^{\delta}_{j}  : \delta \in \widehat{K_0}, 1\leq j\leq d(\delta) \right\rbrace $$ 
becomes an orthonormal basis for $L^{2}(S^{2n-1})$. Given a measurable function $ f $ on $ \C^n $ which  has a well defined restriction on every sphere $S_R = \{ z: |z| =R \}$, we have the spherical harmonic expansion 
\begin{equation} \label{harmonic-exp-1} 
f(R\omega) = \sum_\delta \sum_{j=1}^{d(\delta)} (f_R,Y_j^\delta)_{L^2(S^{2n-1})}  Y_j^\delta(\omega),
\end{equation}
where $ f_R(\omega) := f(R\omega)$ for $\omega \in S^{2n-1}.$ 
If $ z = R \omega, $ we can rewrite the above expansion in the following form. For each $R > 0$, let $ \sigma_R $ be the normalised surface measure on the sphere $S_R = \{ z: |z|=R \}$ defined by 
$$\int_{S_R} f(z) d\sigma_R = \int_{S^{2n-1}} f(R\omega) d\sigma.$$
Denoting the inner product in $ L^2(S_R,d\sigma_R) $ by $ (f,g)_R$, we can rewrite (\ref{harmonic-exp-1}) as
\begin{equation} \label{harmonic-2}
f(z) = \sum_\delta \sum_{j=1}^{d(\delta)} R^{-(a+b)} (f,P_j^\delta)_R   \,\, R^{-(a+b)} P_j^\delta(z), \quad \textup{for } z \in S_R. 
\end{equation} 

There is an operator analogue of the above mentioned spherical harmonics which we briefly recall below. For details, we refer to \cite{Geller}, \cite{Thangavelu1}, and Section 2.7 of \cite{thangavelu}. For each $k \in \mathbb{N}$, consider the the following sub-collection
$$\widehat{K}(k) = \left\lbrace \delta_{a,b}\in \widehat{K_0}: 0 \leq a \leq k, b \in \mathbb{N}\right\rbrace$$
of $\widehat{K_0}$. Let $E_{k}$ be the finite dimensional subspace of $L^{2}(\mathbb{R}^{n})$ spanned by $\left\lbrace \Phi_{\alpha} : |\alpha| = k \right\rbrace$, where $\Phi_{\alpha}$ are the normalized Hermite functions. Let $\mathcal{O} (E_{k})$ be the space of all bounded linear operators $T : E_{k} \to L^{2}(\mathbb{R}^{n})$. Then, $\mathcal{O} (E_{k})$ is a Hilbert space with respect to the following inner product: 
\begin{equation} \label{def-O(E_k)-inner-prod}
\left(T, S\right)_{k} =  \frac{k!(n-1)!}{(k+n-1)!} \sum_{|\alpha|=k}\left(T \Phi_{\alpha}, S \Phi_{\alpha} \right).
\end{equation}
Recall that we have fixed an orthonormal basis $\{P^{\delta}_{j}$ : $1\leq j\leq d(\delta)\}$ of $\mathcal{H}_{a,b}$.  In an impressive work \cite{Geller} Geller proved that the Weyl correspondence of $ P_j^\delta,$ 
$$ \{G(P^{\delta}_{j}) : \delta\in \widehat{K}(k), 1\leq j \leq d(\delta) \} $$ 
forms an orthogonal system in $\mathcal{O}(E_{k})$ and that every operator $T \in \mathcal{O}(E_{k})$ has the expansion 
\begin{align}\label{Oper-Spherical-Harmonic}
T = \sum_{\delta \in \widehat{K}(k)} \sum_{j=1}^{d(\delta)} (C_\delta(2k+n))^{-2} \left(T, G (P^{\delta}_{j})\right)_{k} G(P^{\delta}_{j})
\end{align}
where $(C_\delta(2k+n))^2 = (G(P_j^\delta), G(P_j^\delta)_k.$ These constants are known explicitly, see \cite{Geller}:
$$(C_\delta(2k+n))^2 = 4^{a+b} 2^{n+a+b-1} \frac{\Gamma(k+n+b)}{\Gamma(k-a+1)} \frac{\Gamma(k+1) \Gamma(n)}{\Gamma(k+n)}.$$
It follows that $\{ (C_\delta(2k+n))^{-1} G(P^{\delta}_{j}) : \delta\in \widehat{K}(k), 1\leq j \leq d(\delta)\}$ forms an orthonormal basis for $ \mathcal{O}(E_k).$

For the convenience of the readers (and also for later use) let us record the above result of Geller in the following form. The Hermite functions $ \Phi_\alpha, \alpha \in \N^n $ form an orthonormal basis for $ L^2(\R^n).$ They are eigenfunctions of the Hermite operator $ H = -\Delta+|x|^2 $ with eigenvalues $ (2|\alpha|+n).$ Given $ f \in L^2(\R^n) $ we let $ P_k $ stand for the orthogonal projection of $ L^2(\R^n) $ onto the eigenspace $ E_k.$ More explicitly,  for any $ k \in \N$, we have $ P_kf = \sum_{|\alpha|=k} (f, \Phi_\alpha) \Phi_\alpha.$ Now, for each $\delta \in \widehat{K}(k)$ and $ 1 \leq j \leq d(\delta) $ we let
$$ S_{j,k}^\delta = (C_\delta(2k+n))^{-1} G(P_j^\delta) P_k.$$
We note that $ S_{j,k}^\delta $ are Hilbert-Schmidt operators on $ L^2(\R^n) $ with unit norm. We have

\begin{theorem}[Geller] The collection $ \{ S_{j,k}^\delta:  k \in \N, \delta \in \widehat{K}(k), 1 \leq j \leq d(\delta) \}$ is an orthonormal basis for the Hilbert space $ \mathcal{S}_2 $ of Hilbert-Schmidt operators on $ L^2(\R^n) $ equipped with the inner product $ (T,S) = tr(S^\ast T).$ Moreover, for any Hilbert-Schmidt operator $ T $ on $ L^2(\R^n) $ we have
\begin{equation} T = \sum_{k=0}^\infty \sum_{\delta \in \widehat{K}(k)} \sum_{j=1}^{d(\delta)} ( T, S_{j,k}^\delta) S_{j,k}^\delta \end{equation}
where the series converges in $ \mathcal{S}_2 $ and we have the identity 
\begin{equation} 
\|T\|_{HS}^2 = \sum_{k=0}^\infty \sum_{\delta \in \widehat{K}(k)} \sum_{j=1}^{d(\delta)} |( T, S_{j,k}^\delta)|^2. 
\end{equation}
\end{theorem}  
 
It is clear that the collection $ S_{j,k}^\delta $ is an orthonormal set in view of the definition of the inner product on $ \mathcal{S}_2 $ and the orthogonality properties of $ G(P_j^\delta).$ The rest of the theorem follows from the fact that $ \|T\|_{HS}^2 = \sum_{k=0}^\infty \| TP_k\|_{HS}^2 $ and Geller's result, see \cite{Geller}. 

\medskip 
For  each fixed $\delta = \delta_{a,b} $ and $1 \leq j \leq d(\delta),$ we can define the operator
\begin{align} \label{Oper-Spherical-Harmonic-M-j-delta} 
T_j^\delta := \sum_{k=a}^{\infty} (C_\delta(2k+n))^{-2} \left(T, G (P^{\delta}_{j})\right)_{k} P_k.
\end{align}
For any $ T \in B(L^2(\R^n)) $ the above series in (\ref{Oper-Spherical-Harmonic-M-j-delta}) converges in the strong operator topology.  Indeed, as $ f = \sum_{k=0}^\infty P_kf $ for any $ f \in L^2(\R^n) $ we only need to check that the sequence  $ (C_\delta(2k+n))^{-2} \left(T, G (P^{\delta}_{j})\right)_{k}$ is bounded. But this is easy to see: by the definition
$$ (T,G(P_j^\delta))_k = \frac{\Gamma(k+1) \Gamma(n)}{\Gamma(k+n)} \sum_{|\alpha|=k} (T\Phi_\alpha, G(P_j^\delta)\Phi_\alpha).      $$
Applying Cauchy-Schwarz inequality and recalling the definition of $ (C_\delta(2k+n))^{2} $ we obtain 
$$ | (T,G(P_j^\delta))_k|^2 \leq \|T\|^2 (C_\delta(2k+n))^{2} $$
after making use of the fact that the dimension of $ E_k $ is $ \frac{\Gamma(k+n)}{\Gamma(k+1) \Gamma(n)}.$ Thus we have the estimate
$$ (C_\delta(2k+n))^{-2}|(T,G(P_j^\delta))_k| \leq \|T\| (C_\delta(2k+n))^{-1} $$
which is clearly bounded in view of Stirling's formula for the Gamma function. The operator norm of $ T_j^\delta $ is given by
$$ \| T_j^\delta \| =  \sup_{k \in \N} (C_\delta(2k+n))^{-2} |(T,G(P_j^\delta))_k| < \infty.$$

\begin{remark}  In view of the above discussion we have the formal expansion
$$  T = \sum_{\delta \in \widehat{K_0}} \sum_{j=1}^{d(\delta)} G(P_j^\delta) T_j^\delta $$ for any $ T \in B(L^2(\R^n)).$  For any $ k \in \mathbb{N}$ we  get back the convergent expansion 
$$ T P_k = \sum_{\delta \in \widehat{K_0}} \sum_{j=1}^{d(\delta)} (T,S_{j,k}^\delta) S_{j,k}^\delta.$$
\end{remark}

%%%%%%%%%%%%%%%%%%%%%%%%%%%%%%%%%%%%%%%%%%%%%%%%%%%%%%%%%%%%%%%%%%%%%%%%%
%%%%%%%%%%%%%%%%%%%%%%%%%%%%%%%%%%%%%%%%%%%%%%%%%%%%%%%%%%%%%%%%%%%%%%%%%

\subsection{Homogeneous operators of degree \texorpdfstring{$0$}{}} \label{subsec-hom-degree-0} 
We are in a position to define  homogeneous operators of degree 0.
If $M \in B(L^{2}(\mathbb{R}^{n})),$ then $M P_k \in \mathcal{O}(E_{k})$ and hence $MP_k$ has the expansion as in (\ref{Oper-Spherical-Harmonic}):  that is, 
\begin{align} \label{Oper-Spherical-Harmonic-M}
M P_k = \sum_{\delta \in \widehat{K}(k)} \sum_{j=1}^{d(\delta)} (C_\delta(2k+n))^{-1} \left(M, G (P^{\delta}_{j})\right)_{k} (C_\delta(2k+n))^{-1} G(P^{\delta}_{j}) P_k. 
\end{align} 
This is the operator analogue of the expansion (\ref{harmonic-2}) which we recall for the convenience of comparison:
\begin{equation*} %\label{harmonic-2}
f(z) = \sum_\delta \sum_{j=1}^{d(\delta)} R^{-(a+b)} (f,P_j^\delta)_R   \,\,R^{-(a+b)} P_j^\delta(z), \quad \textup{for } z \in S_R. 
\end{equation*} 
From the above expansion we infer that a function $ f $ on $ \C^n $ is homogeneous of degree $ 0 $ if and only if the coefficients $ R^{-(a+b)} (f,P_j^\delta)_R $ are independent of the radius $ R.$ As observed by Geller \cite{Geller} the operator analogue of `restriction to a sphere of radius $ R $' is just $ MP_k.$ Also $ G(P_j^\delta) $ play the role of solid harmonics.  Thus a comparison of the above two expansions motivates us to make the following definition.

\begin{definition} We say that a bounded linear operator $ M $ on $ L^2(\R^n) $ is homogeneous of degree $ 0 $ if  for any $ \delta \in \widehat{K_0} $ the coefficients $(C_\delta(2k+n))^{-1} \left(M, G (P^{\delta}_{j})\right)_{k}$ are independent of $ k.$ 	
\end{definition}

Thus  for each $\delta = \delta_{a,b}$ and $1 \leq j \leq d(\delta)$, there exist constants $B^{\delta}_{j}$, independent of $k$, such that
\begin{equation} \label{Def-Hom}
(C_\delta(2k+n))^{-1} \left(M, G (P^{\delta}_{j})\right)_{k} =  B^{\delta}_{j}.
\end{equation}
For such operators we then have the expansion
\begin{equation}  MP_k = \sum_{\delta \in \widehat{K}(k)} \sum_{j=1}^{d(\delta)}  B_j^\delta   (C_\delta(2k+n))^{-1}  G(P_j^\delta)P_k, \end{equation}
and  the Hilbert-Schmidt norm of $ MP_K $ turns out to be  independent of $ k.$ More precisely,  
$$\| MP_k\|_{HS}^2  = \sum_{\delta \in \widehat{K}(k)}\sum_{j=1}^{d(\delta)} |B_j^\delta|^2 .$$ 
The above expansion for $ MP_k $ suggests that we introduce  
\begin{align}\label{c-delta-H-def}
C_\delta(H) = \sum_{k=a}^\infty C_\delta(2k+n) P_k, 
\end{align}
as a densely defined unbounded operator using which we represent $ M $ as 
\begin{align}\label{M-1}
M = \sum_{\delta \in \widehat{K}(k)} \sum_{j=1}^{d(\delta)} B_j^\delta G(P_j^\delta)C_\delta(H)^{-1}.
\end{align}
In the above, $ C_\delta(H)^{-1} = \sum_{k=a}^\infty (C_\delta(2k+n))^{-1}  P_k$ is a bounded operator on $L^2(\mathbb{R}^n)$ and the series converges in the strong operator topology. This is easily seen using the fact that for any $ f \in L^2(\R^n) $ the partial sums $ \sum_{j=0}^l P_j f $ converge to $ f $ in $L^2$-norm. We remark that the series converges in the Hilbert-Schmidt norm if $ M $ is Hilbert-Schmidt. \\

From the explicit formula for $(C_\delta(2k+n))^2 $ we infer, via Stirling's formula for the Gamma function, that $(C_\delta(2k+n))^{-1}$ behaves like $ (2k+n)^{-(a+b)/2} $ as $ k $ tends to infinity. Consequently, the operator $ C_\delta(H)^{-1} $ is comparable to the simpler operator $ H^{-(a+b)/2} .$ We will make this more precise in the next section by studying $ C_\delta(H) H^{-(a+b)/2},$ see Lemma \ref{lem-piece-transfer}.

%%%%%%%%%%%%%%%%%%%%%%%%%%%%%%%%%%%%%%%%%%%%%%%%%%
%%%%%%%%%%%%%%%%%%%%%%%%%%%%%%%%%%%%%%%%%%%%%%%%%%
%%%%%%%%%%%%%%%%%%%%%%%%%%%%%%%%%%%%%%%%%%%%%%%%%%
%%%%%%%%%%%%%%%%%%%%%%%%%%%%%%%%%%%%%%%%%%%%%%%%%%
%%%%%%%%%%%%%%%%%%%%%%%%%%%%%%%%%%%%%%%%%%%%%%%%%%

\section{Weyl multipliers on Sobolev spaces} \label{sec-Weyl-mult-proof} 
In this section we prove Theorem \ref{Main-Weyl}, our main theorem stated in the introduction.  Let $ M $ be a bounded linear operator which is homogeneous of degree 0. Under the assumption that $ M $ defines a Weyl multiplier on the Sobolev space $ W_L^{N,1}(\C^n) $ we will prove that $ M = c I $ for some constant $c$. The proof follows in several steps. We begin with the following reduction. \\

Let us denote the action of $ U(n) $ on functions $f$ by $ R_\sigma f(z) = f(\sigma^{-1} z)$. From Stone-von Neumann theorem, we know that for every $\sigma\in U(n)$, there exists a unitary operator $\mu(\sigma)$ on $L^{2}(\mathbb{R}^{n})$ such that 
$$\pi(\sigma z,t) =\mu(\sigma) \pi(z,t) \mu(\sigma)^{*}.$$
For more about these operators we refer to \cite{Folland} and \cite{thangavelu}.
It follows from the definition of the Weyl transform that 
\begin{align} 
W(R_\sigma f) = \mu(\sigma) W(f) \mu(\sigma)^\ast.
\end{align}
This motivates us to define an action of $ U(n) $ on $ B(L^2(\R^n)) $ as follows: Given $ M \in B(L^2(\R^n))$, we define
$$ R_\sigma M := \mu(\sigma) M\mu(\sigma)^\ast.$$ 
As $T_M$ is the operator defined by the relation  $ W(T_M f) = M W(f),$ we easily verify that 
$$ T_{R_\sigma M} f =  R_\sigma T_M R_{\sigma^{-1}} f.$$ 
For any  $ \delta \in \widehat{K_0} $ let $ \chi_\delta $ denote its character.  We define 
\begin{equation} \label{ope-delta-piece}
M^\delta :=  \int_{U(n)} \chi_{\delta}(\sigma^{-1}) R_\sigma M \, d\sigma
\end{equation}
where $ d\sigma $ is the normalised Haar measure on the compact Lie group $ U(n).$
Then, one can verify that 
\begin{equation} \label{op-delta-piece-4} 
T_{M^\delta}f  = \int_{U(n)} \chi_{\delta}(\sigma^{-1}) T_{R_\sigma M} f \, d\sigma = \int_{U(n)} \chi_{\delta}(\sigma^{-1})  R_\sigma (T_M  R_{\sigma^{-1}} f) \, d\sigma. 
\end{equation}

\begin{lemma} \label{Weyl-rot}
Let $M \in B(L^2(\mathbb{R}^n))$ be a Weyl multiplier on $ W^{N,p}_{L}(\mathbb{C}^{n}).$ Then so is $M^{\delta}$  for any $\delta \in \widehat{K_0}$. 
\end{lemma} 
\begin{proof} 
In view of (\ref{op-delta-piece-4}), we only need to show that $ R_\sigma $ is bounded on $ W_L^{N,p}(\C^n)$ uniformly in $ \sigma.$
The result is  true for $N=0$ since $||R_{\sigma}f||_{L^p} = ||f||_{L^p}$ due  the rotation invariance of the Lebesgue measure. For the general case we only need to check that $ \|Z_jR_\sigma f\|_p \leq 
C \sum_{k=1}^n \| Z_kf\|_p $ and $ \| \overline{Z}_j R_\sigma f\|_p \leq C  \sum_{k=1}^n \|\overline{Z}_k f\|_p $ where $ C $ is independent of $ \sigma.$ (A long-winding iteration is needed for the case $ N > 1$ so we restrict ourselves to $ N =1.$) If we let $\sigma^{-1} =(a_{ij})_{n\times n},$ an easy calculation shows that
$$\frac{\partial R_{\sigma}f}{\partial z_{j}}(z) = \sum_{k=1}^{n} a_{kj} \frac{\partial f}{\partial z_k}(\sigma^{-1} z).$$
Moreover, if we let $ w = \sigma^{-1}   z $ then it follows that $ \overline{z}_j = \overline{(\sigma  w)}_j = \sum_{k=1}^n a_{kj} \overline{w}_k $ as $ \sigma $ is unitary.
Therefore,  $ \|Z_{j}R_{\sigma}f\|^{p}_p $ is given by 
$$  \int_{\mathbb{C}^{n}} \left|\sum_{k=1}^{n}a_{kj}\frac{\partial f}{\partial z_{k}}(\sigma^{-1} z) -\frac{\overline{z}_{j}}{4}f(\sigma^{-1} z)\right|^p dz 
=\int_{\mathbb{C}^{n}} \left|\sum_{k=1}^{n}a_{kj}\frac{\partial f}{\partial w_{k}}(w) -\sum_{k=1}^{n}\frac{a_{kj}\overline{w}_{k}}{4}f(w)\right|^p  dw .$$
which is clearly bounded by
$$ \int_{\C^n}  \left(  \sum_{k=1}^n |a_{kj}| |Z_k f(w)| \right)^p dw \leq \left( \sum_{k=1}^{n}|a_{kj}| \| Z_{k}f\|_p  \right)^p \leq C \left( \sum_{k=1}^{n} \| Z_{k}f\|_p  \right)^p . $$
This proves the claim for $ Z_j R_\sigma f.$ The proof for the term $ \overline{Z}_j R_\sigma f $ is similar.
This completes the proof of the lemma.
\end{proof}

When $ M $ is  homogeneous of degree zero, it is easy to calculate $M^\delta$ in terms of $ M.$ In fact in  the expansion of $ M$ given by (\ref{M-1}) we only need to pick up those terms involving $ P_j^\delta.$

\begin{lemma} \label{lem-M-delta-hom}
Let $M \in B(L^2(\mathbb{R}^n))$ be homogeneous of degree zero having the expansion (\ref{M-1}).  Then $M^{\delta}$, defined by (\ref{ope-delta-piece}), has the following simpler form: 
\begin{align} \label{M-delta-hom-0}
M^{\delta}=\sum_{j=1}^{d(\delta)} B^{\delta}_{j} G \left(P^{\delta}_{j} \right) C_\delta(H)^{-1}. 
\end{align}
Consequently, we note that every $ M^\delta $ is homogeneous of degree 0.
\end{lemma}
\begin{proof} It is enough to show that for every $ k \in \N $ we have the equality
\begin{align*}
M^\delta P_k =  \sum_{j=1}^{d(\delta)}  B^{{\delta}}_{j}   (C_\delta(2k+n))^{-1}  G(P^{{\delta}}_{j}) P_k
\end{align*}
which is the same as saying that $ (M^\delta, S_{j,k}^\delta) = B_j^\delta $  and $ (M^\delta, S_{j,k}^\rho) = 0 $ if $ \rho $ is not unitarily equivalent to $ \delta.$ (Here the inner products are taken in $ \mathcal{S}_2.$) We make use of the fact that $ \mu(\sigma) $ commutes with Hermite projections $ P_k$, see \cite{Folland} or \cite{thangavelu}. Consequently, as 
$$ (R_\sigma M)P_k = \mu(\sigma)M \mu(\sigma)^\ast P_k = \mu(\sigma) MP_k  \mu(\sigma)^\ast = R_\sigma (M P_k) $$
we see that $ M^\delta P_k = (MP_k)^\delta $ and so we need to prove
\begin{align*}
(M P_k)^\delta =  \sum_{j=1}^{d(\delta)}  B^{{\delta}}_{j}   (C_\delta(2k+n))^{-1}  G(P^{{\delta}}_{j}) P_k
\end{align*}
for which we consider the inner products
$$ ((MP_k)^\delta, S_{j,k}^\rho) = \int_{U(n)} \chi_{\delta}(\sigma^{-1}) (R_\sigma( MP_k), S_{j,k}^\rho)  \, d\sigma.$$
As operators $ \mu(\sigma)$ are unitary, we see that
$$ (R_\sigma( MP_k), \, S_{j,k}^\rho) = tr( (S_{j,k}^\rho)^\ast \mu(\sigma)MP_k \mu(\sigma)^\ast )  = (MP_k, \, \mu(\sigma)^\ast S_{j,k}^\rho \mu(\sigma)).$$
Thus we are left with calculating the integral
\begin{equation} 
\int_{U(n)} \chi_{\delta}(\sigma^{-1})  \mu(\sigma)^\ast  S_{j,k}^\rho \mu(\sigma) \, d\sigma. 
\end{equation}
At this point we make one more property of the Weyl correspondence, namely $ G(P_j^\rho) P_k= c_n(a^\prime, b^\prime) W(P_j^\rho \varphi_{k-a^\prime}^{n+a^\prime+b^\prime-1})$ where $a^\prime, b^\prime$ are the parameters associated to $ \rho $ and $\varphi_k^{a^\prime+b^\prime+n-1}$ are Laguerre functions of type $(a^\prime+b^\prime+n-1).$ This is a consequence of the Hecke-Bochner identity for the Weyl transform. We refer to equation (2.6.22) in \cite{thangavelu}. As the Laguerre functions are radial, it follows that the required integral is a constant multiple of the Weyl transform of the function 
$$ \varphi_{k-a^\prime}^{a^\prime + b^\prime+n-1}(z) \int_{U(n)}  \chi_{\delta}(\sigma^{-1}) P_j^\rho(\sigma^{-1} z) \, d\sigma.$$ 
But the above integral vanishes unless $\rho$ is unitarily equivalent to $\delta$ in which case it is $P_j^\delta$. This completes the proof of the lemma.
\end{proof}

So far we have proved that if a bounded linear operator $ M $ which is homogeneous of degree 0 defines a Weyl multiplier on the Sobolev space $ W_L^{N,p}(\C^n) $ then for any $ \delta \in \widehat{K_0}$ the operator $ M^\delta $ also defines a multiplier on the same space. Moreover, the operator $ M^\delta $ has a simpler form, namely
\begin{equation} \label{M-delta} 
M^\delta = \sum_{j=1}^{d(\delta)} B_j^\delta G(P_j^\delta) C_\delta(H)^{-1} = G(P_\delta) C_\delta(H)^{-1} 
\end{equation}
where $ P_\delta = \sum_{j=1}^{d(\delta)} B_j^\delta P_j^\delta $ is an element of $ \mathcal{H}_\delta = \mathcal{H}_{a,b}.$ At this stage we would like to replace the operator $ C_\delta(H)^{-1} $ by the simpler operator $ H^{-(a+b)/2}.$ Our goal is to establish the following result. 

\begin{lemma} \label{lem-piece-transfer}
Suppose $ M^\delta $ given in (\ref{M-delta}) defines a Weyl multiplier on the Sobolev space $ W_L^{N,p}(\C^n), 1 \leq p <\infty,$ then so does the operator $ G(P_\delta)H^{-(a+b)/2}.$
\end{lemma}
\begin{proof} By writing $$ G(P_\delta)H^{-(a+b)/2} = G(P_\delta) C_\delta(H)^{-1} C_\delta(H) H^{-(a+b)/2}$$ we only need to show that $ C_\delta(H) H^{-(a+b)/2} $ defines  a Weyl multiplier on all Sobolev spaces $ W_L^{N,p}(\C^n).$  Recall that $ C_\delta(2k+n) $ is a constant multiple of 
\begin{align*}
\left(\frac{\Gamma(k+n+b) \, \Gamma(k+1)}{\Gamma(k-a+1) \, \Gamma(k+n)}\right)^{1/2} = \left((k+n+b-1) \cdots (k+n) k \cdots (k-a+1)\right)^{1/2}. 
\end{align*} 
where $k \geq a$ and if any of $a$ or $b$ is $0$ then corresponding to that no term appears in the above expression. Thus we are led to consider the operator 
$$B_{\delta}(H) = \sum_{k = a}^\infty \left( \frac{k+n+b-1}{2k+n} \times \cdots \times \frac{k+n}{2k+n} \times \frac{k}{2k+n} \times \cdots \times \frac{k-a+1}{2k+n} \right)^{1/2} P_k.$$
As the corresponding Weyl multiplier  is a product of several  multipliers, it is enough to treat individual factors separately.
	
We first consider the multipliers given by  $ m_j(H) = \sum_{k=a}^\infty \left( \frac{k+j}{2k+n} \right)^{1/2} P_k$ where $ n \leq j \leq (n+b-1).$ The corresponding operator $ T_{m_j(H)} $ is a twisted convolution operator $ K_j \times f $, where the kernel $K_j$ is given by 
$$ K_j(z) = (2\pi)^{-n} \sum_{k=a}^\infty \left( \frac{k+j}{2k+n} \right)^{1/2} \varphi_k^{n-1} (z).$$ 
Since $ Z_j (g \times f) = g \times Z_j f $ and $ \overline{Z}_j (g \times f) = g \times \overline{Z}_j f $ we only need to show that $ K_j $ are integrable. We rewrite  the kernel as
$$ K_j(z) = C_n \sum_{k=a}^\infty \left( \frac{2k+2j}{2k+n}\right)^{1/2} \varphi_k^{n-1}(z) = C_n \sum_{k=a}^\infty \left(1+ \frac{2j-n}{2k+n}\right)^{1/2} \varphi_k^{n-1}(z).$$ 
Note that if we let $ d_j = 2j-n$ then $  n \leq d_j \leq n+2(b-1).$ Consider the kernel
\begin{align*}
& \sum_{k=a}^\infty \left(  1+ \frac{d_j}{2k+n}\right) ^{1/2} \varphi_k^{n-1}(z) \\ 
&= \sum_{k=0}^\infty \left(  1+ \frac{d_j}{2k+n}\right) ^{1/2} \varphi_k^{n-1}(z)- \sum_{k=0}^{a-1} \left(  1+ \frac{d_j}{2k+n}\right) ^{1/2} \varphi_k^{n-1}(z).
\end{align*} 
The second term is an integrable function and hence defines a bounded operator on $ L^1(\C^n) $. The operator corresponding to the infinite sum is 
$$ \left( I + \frac{d_j}{L}\right) ^{1/2} = I + \frac{1}{2} d_j L^{-1/2} \int_0^1 (L+td_j)^{-1/2} dt$$ 
where $ L $ is the twisted Laplacian. As $ L^{-1/2} $ and $ (L+t d_j)^{-1/2} $ are  given by integrable kernels (see Proposition \ref{prop-inverse-frac-twisted-lap}) we are done. 

For the remaining terms corresponding to the factors $ \left(  \frac{k-j}{2k+n}\right) ^{1/2} $, with $ 0 \leq j \leq a-1 $, the above idea fails and  we have to use the power series expansion. As earlier we only need to show that the kernel of each such operator is integrable. With large $ l $ to be chosen later we split the kernel into two parts: 
$$ \sum_{k=a}^\infty \left( \frac{k-j}{2k+n} \right) ^{1/2} \varphi_k^{n-1} = \sum_{k=l}^\infty \left( \frac{k-j}{2k+n} \right) ^{1/2}  \varphi_k^{n-1} + \sum_{k=a}^{l-1} \left( \frac{k-j}{2k+n} \right) ^{1/2}  \varphi_k^{n-1} .$$ 
Once again the second term is an integrable function and hence defines a bounded operator on $ L^1(\C^n) $, so we only need to analyse the kernel given by the infinite sum. Note that 
$$ \left(  \frac{k-j}{2k+n}\right) ^{1/2}  = 2^{-1/2} \left(  1- \frac{n+2j}{2k+n}\right) ^{1/2} $$ 
where $ 0 \leq j \leq a-1 $ and $ k \geq l \geq a$ so that  $ 0 < \frac{n+2j}{2k+n} <1.$ We make use of the power series expansion 
\begin{align} \label{expansion-half-power}
(1-d)^{1/2} = 1- \sum_{i=1}^\infty c_i d^i
\end{align} 
where $ c_i $ are explicitly given by
$$ c_i = \frac{\Gamma(2i+1)}{2^{2i} (\Gamma(i+1))^2 (2i-1)}.$$
Note that the above series converges for all $ |d| \leq 1.$ 

Therefore, writing $c_0 = -1$, we can expand 
$$ \left( 1- \frac{n+2j}{2k+n}\right)^{1/2} = - \sum_{i=0}^\infty c_i  (n+2j)^{i} (2k+n)^{-i} $$
and we are led to consider the operators 
$$ \sum_{k=l}^\infty \left( \frac{k-j}{2k+n}\right) ^{1/2} \varphi_k^{n-1} \times f = - 2^{-1/2} \sum_{i=0}^\infty c_i (n+2j)^{i}  T_{l,i} f $$
where the operators $ T_{l,i} $ are defined by
$$ T_{l,i} f = \sum_{k=l}^\infty (2k+n)^{-i} \varphi_k^{n-1} \times f.$$ 
For each fixed $i$, the operator $T_{l,i}$ can be estimated easily. For this note that 
\begin{align*}
T_{l,i} f &= \sum_{k=0}^\infty (2k+n)^{-i} \varphi_k^{n-1} \times f - \sum_{k=0}^{l-1} (2k+n)^{-i} \varphi_k^{n-1} \times f \\ 
&= L^{-i} f - \sum_{k=0}^{l-1} (2k+n)^{-i} \varphi_k^{n-1} \times f
\end{align*}
which implies that $T_{l,i}$ is a bounded operator on $ L^1(\C^n) $ (in view of Proposition \ref{prop-inverse-frac-twisted-lap}). As a consequence, the finite sum of operators $ \sum_{i=0}^{l-1} c_i (n+2j)^i T_{l,i} f $ is bounded on $ L^1(\C^n) $. So, we concentrate on estimating $ T_{l,i} $ for $ i\geq l.$ If $ K_{l,i} $ stands for the kernel of $ T_{l,i} $ we see that 
$$ \| K_{l,i} \|_1 \leq (2l+n)^{-i/2} \sum_{k=l}^\infty (2k+n)^{-i/2} \|\varphi_k^{n-1}\|_1. $$ 
We choose $ l > a^2 $ large enough so that the above series converges for $ i \geq l$ and $ (n+2j) < (2l+n)^{1/2}.$ Consequently,
$$ \left\| \sum_{i=l}^\infty c_i (n+2j)^i K_{l,i} \right\|_1 \leq \sum_{k=l}^\infty (2k+n)^{-l/2} \|\varphi_k^{n-1}\|_1 .$$
This completes the proof of the lemma. 
\end{proof}

Starting from a Weyl multiplier $ M $ on $ W_L^{N,p}(\C^n) $ we have shown that  the operator $ G(P_\delta) H^{-(a+b)/2} ,$ where $ P_\delta $ is a particular element of $ \mathcal{H}_{a,b} ,$ is also a Weyl multiplier. We now show that the same is true of $ G(P) H^{-(a+b)/2} $ for any element $ P \in   \mathcal{H}_{a,b}.$

\begin{lemma} \label{lem-solid-harmonic-equiv}
Suppose for some $ P \in \mathcal{H}_{a,b} $ the operator $ G(P) H^{-(a+b)/2} $ defines a Weyl multiplier on $ W_L^{N,p}(\C^n).$  Then  the same is true for any $ Q \in \mathcal{H}_{a,b}.$
\end{lemma}
\begin{proof} 
In proving this result we make use of the fact that $ \delta $ is an irreducible unitary representation of the compact Lie group $ U(n) $ on the finite dimensional Hilbert space $ \mathcal{H}_{a,b}.$ 
Hence given  $ P, Q  \in \mathcal{H}_{a,b},$ we can express  $Q $  as a finite linear combination of $\delta(\sigma_j) P$, where $\sigma_j$ varies over $U(n)$ (see \cite{JST}, \cite{Sanjay-Thangavelu}). Thus we have $ Q(z) = \sum_{j=1}^l c_j P(\sigma_j^{-1} z).$  In view of the formula $ G(Q)P_k = c_n(a,b) W(Q \varphi_{k-a}^{n+a+b-1})$ we obtain
$$ G(Q) P_k = c_n(a,b) \sum_{j=1}^l c_j  W \left( \delta(\sigma_j) P \varphi_{k-a}^{n+a+b-1} \right) = \sum_{j=1}^l c_j \mu(\sigma_j) G(P) \mu(\sigma_j)^\ast P_k$$
where we have used the fact that $ \varphi_k^{n+a+b-1} $ is radial and  $ P_k $ commutes with $ \mu(\sigma) $ for any $ \sigma \in U(n).$ As the above is true for any $ k $ we have proved
$$ G(Q) H^{-(a+b)/2} = \sum_{j=1}^l c_j  \mu(\sigma_j) G(P) H^{-(a+b)/2} \mu(\sigma_j)^\ast. $$
Therefore, if $ R(P) $ and $ R(Q) $ stand for the operators with multipliers $ G(P)H^{-(a+b)/2} $ and $ G(Q)H^{-(a+b)/2} $ we have the relation
$$  W(R(Q)f) = \sum_{j=1}^l c_j  \mu(\sigma_j) G(P)H^{-(a+b)/2}  \mu(\sigma_j)^\ast W(f) .$$
Recalling that $ W(R_\sigma f) = \mu(\sigma) W(f) \mu(\sigma)^\ast$ the above simply means that
$$  R(Q) f  = \sum_{j=1}^l c_j  R_{\sigma_j} R(P) R_{\sigma_j^\ast} f.$$
As shown in the proof of Lemma \ref{Weyl-rot}, the operators  $ R_\sigma $ are bounded on Sobolev spaces. Consequently, from the above we conclude that $ R(Q) $ is bounded on a Sobolev space whenever $R(P) $ is bounded.
This proves the lemma.
\end{proof}

We now consider the Riesz transform $ R(Q) $ where $ Q(z) = z_j^a \bar{z_k}^b $ where $ j \neq k.$ (When $ n=1,$ we take $ Q(z) = z^a $ or $ \bar{z}^b.$) In this case $G(Q) $ is explicitly known. Indeed, as proved in \cite{Geller} (see also Proposition 2.7.4 in \cite{thangavelu}) we have, for any $ \lambda \in \R^\ast$
$$ G_\lambda(z_j^a \bar{z}_k^b) = c(a,b) \lambda^{-a-b}  (A_k(\lambda)^\ast)^b A_j(\lambda)^a.$$
Thus the boundedness of the Weyl multiplier $ G(P^\delta)H^{-(a+b)/2} $ on a Sobolev space implies the boundedness of the higher order Riesz transforms $ R_{jk} $ corresponding to the multipliers $ (A_k^\ast)^b A_j^a H^{-(a+b)/2}.$ More generally, we can check that the same is true of $ R_{jk}(\lambda)$ with multiplier $(A_k(\lambda)^\ast)^b A_j(\lambda)^a H(\lambda)^{-(a+b)/2}.$
For any $ \lambda > 0 $ we have the relations $ W_\lambda(Z_k(-\lambda)f) = \frac{i}{2} (A_k(\lambda))^\ast W_\lambda(f) $ and $ W_\lambda(\overline{Z}_j (-\lambda)f) = \frac{i}{2} A_j(\lambda) W_\lambda(f) $ and hence
$$ R_{jk}(\lambda) = Z_k(-\lambda)^b \overline{Z}_j (-\lambda)^a L(-\lambda)^{-(a+b)/2}.$$
The next result shows that these Riesz transforms converge to Euclidean Riesz transforms on $ \R^{2n} $ as $ \lambda $ decreases to zero.

\begin{lemma} \label{conv-scaled-riesz}
Let $ f $ be a Schwartz function on $ \C^n. $  Then  for any $\alpha, \beta \in \N^n, $ there exists a constant $ C_{\alpha,\beta}, $ depending only on $\alpha, \beta $ and $ n $ such that the following point-wise convergence holds: 
\begin{align}
& \lim_{\lambda \to 0^{+}} Z(-\lambda)^\alpha \overline{Z}(-\lambda)^\beta L(-\lambda)^{-(|\alpha|+|\beta|)/2} f(z) \\
\nonumber & = C_{\alpha,\beta} \left(\frac{\partial}{\partial x} - i \frac{\partial}{\partial y} \right)^{\alpha } \left(\frac{\partial}{\partial x} + i \frac{\partial}{\partial y} \right)^{\beta } \left(-\Delta_{\C^{n}}\right)^{-(|\alpha|+|\beta|)/2} f(z). 
\end{align} 
\end{lemma} 
\begin{proof} 
We let $ s = (|\alpha| + |\beta|)/2 $ and express $ L(-\lambda)^{-s} $ in terms of the associated heat kernel as
\begin{align*}
L(-\lambda)^{-s}f(z) = & \frac{ \lambda^{n}}{ 4^n \pi^n \Gamma(s)}  \int_{0}^{\infty} \int_{\mathbb{C}^{n}} t^{s-1} (\sinh (\lambda t))^{-n} e^{-\frac{1}{4} \lambda \coth (\lambda t)|z-w|^{2}} e^{- i \frac{\lambda}{2} \Im(z\cdot\overline{w})}  f(w) \, dw\, dt.
\end{align*}
Applying Parseval's identity for the Euclidean Fourier transform in $w$-variable of the above expression, we see  that $L(-\lambda)^{-s} f(z)$ is a constant multiple of 
\begin{align} \label{L-negative-dil}
\int_{0}^{\infty} \int_{\mathbb{C}^{n}}t^{s-1} (\cosh (\lambda t))^{-n} e^{- \frac{1}{\lambda} \tanh(\lambda t)|\zeta|^{2}} \widehat{f} \left( \zeta - \frac{i \lambda z}{2} \right) e^{i(\xi,\eta)\cdot(x,y)} \, d\zeta \, dt. 
\end{align}
As $ f $ is a Schwartz function the above integrals converge absolutely and hence we can take the derivatives under the integral sign. We observe that any derivative falling on  
$\widehat{f} \left( \zeta - \frac{i \lambda z}{2} \right)$ brings out a factor of $ \lambda $ which disappears in the limit. Therefore, what matters is when the derivatives fall on the exponential function. As
$$ \lim_{\lambda \rightarrow 0} \frac{1}{\Gamma(s)} \int_{0}^{\infty} t^{s-1} (\cosh (\lambda t))^{-n} e^{- \frac{1}{\lambda} \tanh(\lambda t)|\zeta|^{2}} \, dt = |\zeta|^{-2s} $$ 
we see that
$ Z(-\lambda)^\alpha \overline{Z}(-\lambda)^\beta L(-\lambda)^{-s} f(z) $ converges (upto a constant multiple) to 
$$\int_{\mathbb{R}^{2n}} (\xi - i \eta)^{\alpha} (\xi + i \eta)^{\beta} (|\xi|^2+|\eta|^2)^{-s} \widehat{f}(\xi,\eta) e^{i(\xi,\eta) \cdot (x, y)} \, d\xi \, d\eta$$ 
which is a constant multiple of 
$$\left(\frac{\partial}{\partial x} - i \frac{\partial}{\partial y} \right)^{\alpha} \left(\frac{\partial}{\partial x} + i \frac{\partial}{\partial y} \right)^{\beta} \left( -\Delta_{\C^n} \right)^{-s} f(z).$$

This completes the proof of Lemma \ref{conv-scaled-riesz}. 
\end{proof}

\medskip 
\begin{proof}[Proof of Theorem \ref{Main-Weyl}] 
For the sake of simplicity of notation we give the proof when $ N =1.$ Suppose that a bounded linear operator $ M $ which is homogeneous of degree zero defines a Weyl multiplier on  the Sobolev space $ W_L^{1,1}(\C^n).$ We will show that  for any non trivial $ \delta = \delta_{a,b} \in \widehat{K_0} $ the operator $ M^\delta = 0.$  In view of Lemma \ref{Weyl-rot} we know that $ M^\delta $ also defines a Weyl multiplier on $ W_L^{1,1}(\C^n).$ Appealing to Lemmas \ref{lem-M-delta-hom}, \ref{lem-piece-transfer} and \ref{lem-solid-harmonic-equiv} we conclude that the Riesz transforms $ R_{jk} $ corresponding to the multipliers $  A_{k}^{b} (A_j^\ast)^a H^{-(a+b)/2}$ are bounded on $ W_L^{1,1}(\C^n).$ Using Lemma \ref{conv-scaled-riesz} and Theorem \ref{thm-BP}, we will show that this is not possible unless $ M^\delta =0.$ This will then prove our theorem.

We make use of the following transformation properties of $ Z_j(\lambda), \overline{Z}_j (\lambda) $ and $ L(\lambda) $ under dilations. For $ \lambda > 0 $ let us set $ f_\lambda(z) = f(\lambda^{-1/2}z).$
A simple calculation shows that
$$ \sqrt{\lambda} (Z_j f_\lambda)(\sqrt{\lambda}z) = Z_j(\lambda)f(z),\,\,\sqrt{\lambda} (\overline{Z}_j  f_\lambda)(\sqrt{\lambda}z) = \overline{Z}_j (\lambda)f(z),$$
and hence $ \lambda (L f_\lambda)(\sqrt{\lambda}z) = L(\lambda)f(z).$ As $ L^{-s} $ is defined in terms of the semigroup $ e^{-tL} $ we can also verify that
$$ \lambda^{-(a+b)/2}(L^{-(a+b)/2}f_\lambda)(\sqrt{\lambda} z) =  L(\lambda)^{-(a+b)/2}f(z).$$
Such relations are also valid for the right invariant counterparts $ Z_j(-\lambda) = Z_j^R(\lambda)$ etc. These properties have the implication that the Riesz transform $ R_{jk} $ satisfies 
$$ (R_{jk}f_\lambda)(\sqrt{\lambda} z) =  R_{jk}(\lambda)f(z).$$
Using this along with Lemma \ref{conv-scaled-riesz} we will be able to prove Theorem \ref{Main-Weyl}.

As we are assuming that $ R_{jk} $ is bounded on $ \mathring{W}_L^{1,1}(\C^n) $ ( $\cong W_L^{1,1}(\C^n) $ in view of  Theorem \ref{thm-Weyl-norm-equiv}) we have the following inequality
$$ \sum_{i=1}^n  \int_{\C^n} \left(  | Z_i R_{j,k}f(z) \,\varphi_i(z)| + | \overline{Z}_iR_{jk}f(z) \, \psi_i(z)| \right)  dz \leq C \|f\|_{\mathring{W}_L^{1,1}}  \|g\|_\infty $$
for any sequence $ g =(\varphi_i, \psi_i) $ with $ \|g\|_\infty = \max_{1\leq i\leq n} \{ \|\varphi_i\|_\infty, \|\psi_i\|_\infty\}.$ In the above inequality, we replace $ f $ by $f_\lambda$, $g $ by $ g_\lambda $ and make a 
change of variables to get
$$ \sum_{i=1}^n  \int_{\C^n} \left(  | Z_i(\lambda) R_{j,k}(\lambda)f(z) \,\varphi_i(z)| + | \overline{Z}_i(\lambda)R_{jk}(\lambda)f(z) \, \psi_i(z)| \right)  dz \leq C \|f\|_{\mathring{W}_{L(\lambda)}^{1,1}}  \|g\|_\infty .$$
In obtaining the above we have made use of the relation $ \| f_\lambda\|_{\mathring{W}_L^{1,1}} = \lambda^{n-1/2} \| f\|_{\mathring{W}_{L(\lambda)}^{1,1}}.$  A simple calculation shows that as $ \lambda $ goes to zero, $\|f\|_{\mathring{W}_{L(\lambda)}^{1,1}} $ converges to $\|f\|_{\mathring{W}^{1,1}}$ where $ \mathring{W}^{1,1} $ is the homogeneous Sobolev space on $ \C^n.$ We now take limit as $ \lambda $ decreases to zero. By Fatou's lemma, the result proved in Lemma \ref{conv-scaled-riesz} allows us to conclude
$$\sum_{i=1}^n \int_{\C^n} \left( \left| \frac{\partial}{\partial z_i} T_{j,k} f(z) \, \varphi_i(z) \right| + \left| \frac{\partial}{\partial \bar{z}_i} T_{j,k} f(z) \, \psi_i(z) \right| \right) dz \leq C \|f\|_{\mathring{W}^{1,1}}  \|g\|_\infty.$$
where $ T_{j,k} $ stands for the Euclidean Riesz transform  which corresponds to the Fourier multiplier
$$ m_{j,k}(\xi,\eta) = (\xi_k-i\eta_k)^b(\xi_j+i\eta_j)^a (|\xi|^2+|\eta|^2)^{-(a+b)/2}.$$
The above inequality allows us to conclude that $ T_{j,k} $ is bounded on the homogeneous Sobolev space $ \mathring{W}^{1,1}(\C^n) .$ But now the multiplier $ m_{j,k}(\xi,\eta) $ associated to $ T_{j,k} $ is homogeneous of degree zero and hence by Theorem \ref{thm-BP}, $ T_{j,k} $ cannot be bounded on $ \mathring{W}^{1,1}(\C^n)$ unless $a=b=0$. This contradiction proves that $ M^\delta = 0 $ for any non trivial $ \delta \in \widehat{K_0}.$ Consequently, $ M = c I $ which proves the theorem.
\end{proof}

%%%%%%%%%%%%%%%%%%%%%%%%%%%%%%%%%%%%%%%%%%%%%%%%%%%
%%%%%%%%%%%%%%%%%%%%%%%%%%%%%%%%%%%%%%%%%%%%%%%%%%%
%%%%%%%%%%%%%%%%%%%%%%%%%%%%%%%%%%%%%%%%%%%%%%%%%%%
%%%%%%%%%%%%%%%%%%%%%%%%%%%%%%%%%%%%%%%%%%%%%%%%%%%
%%%%%%%%%%%%%%%%%%%%%%%%%%%%%%%%%%%%%%%%%%%%%%%%%%%
%%%%%%%%%%%%%%%%%%%%%%%%%%%%%%%%%%%%%%%%%%%%%%%%%%%
%%%%%%%%%%%%%%%%%%%%%%%%%%%%%%%%%%%%%%%%%%%%%%%%%%%
%%%%%%%%%%%%%%%%%%%%%%%%%%%%%%%%%%%%%%%%%%%%%%%%%%%
%%%%%%%%%%%%%%%%%%%%%%%%%%%%%%%%%%%%%%%%%%%%%%%%%%%
%%%%%%%%%%%%%%%%%%%%%%%%%%%%%%%%%%%%%%%%%%%%%%%%%%%
%%%%%%%%%%%%%%%%%%%%%%%%%%%%%%%%%%%%%%%%%%%%%%%%%%%
%%%%%%%%%%%%%%%%%%%%%%%%%%%%%%%%%%%%%%%%%%%%%%%%%%%
%%%%%%%%%%%%%%%%%%%%%%%%%%%%%%%%%%%%%%%%%%%%%%%%%%%
%%%%%%%%%%%%%%%%%%%%%%%%%%%%%%%%%%%%%%%%%%%%%%%%%%%
%%%%%%%%%%%%%%%%%%%%%%%%%%%%%%%%%%%%%%%%%%%%%%%%%%%
%%%%%%%%%%%%%%%%%%%%%%%%%%%%%%%%%%%%%%%%%%%%%%%%%%%
%%%%%%%%%%%%%%%%%%%%%%%%%%%%%%%%%%%%%%%%%%%%%%%%%%%
%%%%%%%%%%%%%%%%%%%%%%%%%%%%%%%%%%%%%%%%%%%%%%%%%%%
%%%%%%%%%%%%%%%%%%%%%%%%%%%%%%%%%%%%%%%%%%%%%%%%%%%
%%%%%%%%%%%%%%%%%%%%%%%%%%%%%%%%%%%%%%%%%%%%%%%%%%%

\section{Fourier multipliers on the Heisenberg group} \label{sec-mult-Heisenberg-proof} 
In this section our main goal is to prove Theorem \ref{thm2} stated in the introduction. Recall that we are considering left Fourier multipliers defined by the relation $ \widehat{T_mf}(\lambda) = m(\lambda) \widehat{f}(\lambda) $ for $ f \in L^2(\He^n)$ where $ m(\lambda) $ is a bounded function on $ \R^\ast $ taking values in $ B(L^2(\R^n)).$ Recall that left multipliers are invariant under right translations and they are of the form $ S \ast f $ where $ S $ is a pseudo-measure.

\subsection{Riesz transforms on Sobolev spaces} \label{subsec-Riesz-Sob} 
The prototype of Fourier multipliers that are homogeneous of degree zero are provided by the Riesz transforms. So we take them at first and see if they are bounded on homogeneous or non-homogeneous Sobolev spaces. In this section we show that the Riesz transforms on the Heisenberg group are not bounded on homogeneous or non-homogeneous Sobolev spaces. Recall that for any bigraded solid harmonic $ P $ of bi-degree $ (a,b) $ the Riesz transform $ R_P $ is defined initially on $ L^2(\He^n) $ by the relation
$$ \widehat{R_Pf}(\lambda) = G_\lambda(P)H(\lambda)^{-(a+b)/2} \widehat{f}(\lambda).$$
Note that $ R_P $ are left Fourier multipliers on the Heisenberg group. Our Sobolev spaces $ W^{N,1}(\He^n) $ and $\mathring{W}^{N,1}(\He^n)$ are defined in terms of the left invariant vector fields $ X_j, Y_j $ and hence we have the relations $ X_j R_Pf = R_P(X_jf)$ and $Y_jR_P(f) = R_P(Y_jf) $ for any $ j = 1, 2, \ldots, n.$ These are known to be singular integral operators with kernels $ K_P $ that are homogeneous of degree $ -(2n+2) $  in the sense that $ K_P(\delta_rg) = r^{-(2n+2)} K_P(g) $ where for $ r > 0 $ the non-isotropic dilations are given by $ \delta_r(z,t) = (rz,r^2t).$ In the case of first order Riesz transforms, we use the notation $ R_j f =\tilde{X}_j \tilde{\mathcal{L}}^{-1/2} f $ and $ R_{j+n}f =\tilde{Y}_j \tilde{\mathcal{L}}^{-1/2} f $ for $ j = 1, 2, \ldots, n.$  It is more informative to study a one parameter family of Riesz transforms $ R_j^\varepsilon $ indexed by $ \varepsilon > 0 $ and investigate their behaviour as $ \varepsilon \rightarrow 0.$ These Riesz transforms are associated to sublaplacians $ \tilde{\mathcal{L}}^\varepsilon $ on certain Lie groups $ \He_\varepsilon^n$ which are isomorphic to $ \He^n.$\\

The group $ \He_\varepsilon^n $ is just $ \C^n \times \R $ as a manifold equipped with the group structure
$$ (z,t)(w,s) = (z+w, t+s+\frac{\varepsilon}{2} \Im(z \cdot \bar{w})). $$
Note that $ \He_1^n = \He^n $ and $ \He_0^n$ is the abelian group $\C^n \times \R .$  For each $ \varepsilon > 0 $ the group $ \Hep$ is isomorphic to $ \He^n.$ Indeed, 
$ \varphi_\varepsilon : \He^n \rightarrow \Hep $ given by $ \varphi_\varepsilon(z,t)= (\varepsilon^{-1/2}z,t)$ is an isomorphism. As a Haar measure on $ \Hep $ we take the Lebesgue measure $ dz \, dt .$ Denoting the convolution of two functions $ f $ and $ g $ on $ \Hep $ by $ f\ast_\varepsilon g $ we note that
$$  (f \ast_\vep g) \circ \varphi_\vep =  \vep^{-n} (f \circ \varphi_\vep)\ast (g \circ \varphi_\vep).$$
From the above we make the useful observation that
$$ \vep^{-n} (f \circ \varphi_\vep)\ast (g \circ \varphi_\vep)(\sqrt{\vep}z,t) \rightarrow  f\ast_0 g(z,t) $$
as $\epsilon$ goes to $0$, where $ \ast_0 $ stands for the standard convolution on $ \C^n \times \R.$  We can also easily find out the left (right) invariant vector fields on these groups.\\

Let $ X_j^\vep, Y_j^\vep $ stand for the left invariant vector fields on $ \Hep $ associated to the standard subgroups $ \Gamma_j = \{ (se_j,0): s \in \R \} $ and $ \Gamma_{j+n} = \{ (ise_j,0): s \in \R \} .$ Then a simple calculations shows that
\begin{equation} \label{rel-1} 
(X_j^\vep f)\circ \varphi_\vep = \sqrt{\vep} X_j (f \circ \varphi_\vep), \quad (Y_j^\vep f)\circ \varphi_\vep = \sqrt{\vep}  Y_j (f \circ \varphi_\vep) 
\end{equation}
and a similar relation holds for the right invariant vector fields $ \tilde{X}_j^\vep, \tilde{Y}_j^\vep $. More explicitly, we have
\begin{equation} \label{vfiel} 
X_j^\vep f(z,t) = \frac{\partial}{\partial x_j}+ \frac{\vep}{2}y_j \frac{\partial}{\partial t}, \quad Y_j^\vep f(z,t) = \frac{\partial}{\partial y_j} - \frac{\vep}{2}x_j \frac{\partial}{\partial t}. 
\end{equation}
From these expressions it is clear that these vector fields converge to the standard partial derivatives $ \frac{\partial}{\partial x_j} $ and $ \frac{\partial}{\partial y_j } $ as expected.
Along with these, we also consider the right invariant vector fields $ \tilde{X}_j^\vep, \tilde{Y}_j^\vep $ bearing similar relations with the right invariant vectors fields $ \tilde{X}_j, \tilde{Y}_j $ on $ \He^n.$ We let $ \tilde{\mathcal{L}}_\vep = - \sum_{j=1}^n ( (\tilde{X}_j^\vep)^2+ (\tilde{Y}_j^\vep)^2) $ stand for the right invariant sublaplacian on $ \Hep $ which is related to $ \tilde{\mathcal{L}} $ by
$$  (\tilde{\mathcal{L}}_\vep f)\circ \varphi_\vep =  \vep \tilde{\mathcal{L}}(f\circ \varphi_\vep).$$
This translates into the following relation between the semigroups generated by them
\begin{equation} \label{rel-heat} 
(e^{-t \tilde{\mathcal{L}}_\vep} f)\circ \varphi_\vep = e^{-t \vep \tilde{\mathcal{L}}}(f\circ \varphi_\vep). 
\end{equation}
Consequently, we also have similar relations for fractional powers and Riesz transforms, e.g. $ (R_j^\vep f)\circ \varphi_\vep = R_j(f\circ \varphi_\vep) $ for any $ j.$ We make use of these relations in what follows.\\

For $ j =1, 2, \ldots,n $, we define $ R_j^0 = \frac{\partial}{\partial x_j} (-\Delta_{\C^n})^{-1/2} $ and $ R_{j+n}^0 = \frac{\partial}{\partial y_j} (-\Delta_{\C^n})^{-1/2} $ which are the Riesz transforms on $ \C^n.$  We treat them as operators acting on $ L^2(\He^n).$ We would like to prove that $ R_j^\vep $ converge to $ R_j^0 $ as $ \vep $ goes to zero in a suitable sense by making use of the fact that convolution on $ \Hep $ converges to the Euclidean convolution on $ \C^n \times \R.$ As the Riesz transforms are principal value convolutions, we consider the operators $ \tilde{X}_j^\vep (\gamma^2+\tilde{\mathcal{L}}_\vep)^{-1/2}$ and $ \tilde{Y}_j^\vep (\gamma^2+\tilde{\mathcal{L}}_\vep)^{-1/2}$ where $ \gamma > 0 $ and study their convergence as $ \vep $ goes to zero. We first consider the operators  $  (\gamma^2+\tilde{\mathcal{L}}_\vep)^{-1/2}.$

\begin{proposition} \label{fractional-sublap-convergence}
Let $ \gamma >0.$ Then $ (\gamma^2+\tilde{\mathcal{L}}_\vep)^{-1/2}$ are uniformly bounded on $ L^p(\He^n) $ for any $ 1 \leq p < \infty $ and converge to $  (\gamma^2-\Delta_{\C^n})^{-1/2} $ in the norm as $ \vep $ goes to zero. Moreover, when $ f $ is a Schwartz function we also have the pointwise convergence.
\end{proposition}
\begin{proof} 
Let $ K_{\gamma,\vep} $ be the convolution kernel of $ (\gamma^2+\tilde{\mathcal{L}}_\vep)^{-1/2}$ which can be expressed in terms of the heat kernel $ p_{\eta ,\vep} $ associated to $ \tilde{\mathcal{L}}_\vep$ as follows:
$$ K_{\gamma,\vep}(z,t) = \frac{1}{\sqrt{\pi}} \int_0^\infty \eta^{-1/2} e^{-\eta \gamma^2} p_{\eta ,\vep}(z,t) \, d\eta.$$
When $ \vep = 1$ we simply write $ K_\gamma $ for the kernel of $ (\gamma^2+\tilde{\mathcal{L}})^{-1/2}.$ The relation \eqref{rel-heat} between the semigroups allows us to conclude that
$$  (p_{\eta ,\vep} \ast_\vep f)\circ \varphi_\vep = p_{\eta \vep} \ast (f \circ \varphi_\vep)$$ and consequently, we have the relation
\begin{equation} \label{rel-2} 
(K_{\gamma,\vep} \ast_\vep f)\circ \varphi_\vep = \vep^{-1/2}  K_{\gamma \vep^{-1/2}} \ast (f \circ \varphi_\vep). 
\end{equation}
Now a simple calculation shows that 
$ (K_{\gamma \vep^{-1/2}} \ast (f \circ \varphi_\vep))\circ \varphi_\vep^{-1}(z,t) $ is given by the integral
$$ \vep^{n+1} \int_{\He^n}  K_{\gamma \vep^{-1/2}}(\sqrt{\vep}w,\vep s) f(z-w, t-\vep s - \frac{\vep}{2} \Im(z \cdot \bar{w}) ) \, dw \, ds.$$
We now make use of the homogeneity property of the heat kernel $ p_\eta (z,t) $ under the non-isotropic dilations $ \delta_r(z,t) = (rz,r^2t).$ From the explicit formula (\ref{heat}) we see that
$ p_\eta (\sqrt{\vep}z,\vep t) = \vep^{-n-1} p_{\eta /\vep} (z,t) $ and hence
$$ K_{\gamma \vep^{-1/2}}(\sqrt{\vep}w,\vep s) =  \vep^{-n-1} \frac{1}{\sqrt{\pi}} \int_0^\infty \eta^{-1/2} e^{-\gamma^2 (\eta / \vep)} p_{\eta /\vep}(w,s) \, d\eta = \vep^{-n-1/2} K_\gamma(w,s).$$
Making use of this we finally arrive at the relation
\begin{equation}\label{con} K_{\gamma,\vep} \ast _\vep f(z,t) = \int_{\He^n}  K_{\gamma}(w,s) f(z-w, t- \vep s - \frac{\vep}{2} \Im(z \cdot \bar{w}) ) \, dw \, ds. 
\end{equation}

Now, the kernel $ K_\gamma $ is positive, integrable on $ \He^n $ and we have
$$ K_\gamma^\lambda(w) = C_n \int_0^\infty \eta^{-1/2} e^{- \eta \gamma^2} \left( \frac{\lambda}{\sinh(\lambda \eta)}\right)^n e^{-\frac{1}{4}\lambda (\coth(\eta \lambda))|w|^2} \, d\eta.$$
In particular, the integral of $ K_\gamma $ in the central variable is given by
$$ K_\gamma^0(w)  = C_n \int_0^\infty \eta^{-n-1/2} e^{- \eta \gamma^2} e^{-\frac{1}{4 \eta} |w|^2} \, d\eta $$
which can be calculated explicitly in terms of the Macdonald function. For our purposes it suffices to note that
$$ \int_{\C^n} e^{-i\Re(z\cdot \bar{w})} K_\gamma^0(w) \, dw  = C_n (\gamma^2+|z|^2)^{-1/2} $$ 
which shows that $ K_\gamma^0(w)$ is a constant multiple of the kernel of $ (\gamma^2-\Delta_{\C^n})^{-1/2}.$ Let us denote by $  g \ast_{\C^n} f $ the convolution of $ g $ with $ f $ in the $ \C^n $-variable. Returning to equation \eqref{con} we see that $K_{\gamma,\vep} \ast _\vep f(z,t) - K_\gamma^0 \ast_{\C^n} f(z,t) $ is given by
$$ \int_{\He^n} K_{\gamma}(w,s) \left( f(z-w, t- \vep s - \frac{\vep}{2} \Im(z \cdot \bar{w})) - f(z-w,t) \right) dw \, ds.$$
It is now clear that when $ f $ is Schwartz, the above converges pointwise to zero as $\epsilon$ goes to $0$. By Minkowski's integral inequality we also have
$$ \| K_{\gamma,\vep} \ast_\vep f - K_\gamma^0 \ast_{\C^n} f \|_p \leq \int_{\He^n}  K_{\gamma}(w,s) F(\vep w,\vep s) \, dw \, ds$$
where 
$$ F(w,s) = \left( \int_{\He^n} \left| f(z, t-s - \frac{1}{2} \Im(z\cdot \bar{w})) - f(z,t) \right|^p dz \, dt \right)^{1/p}. $$
In order to prove the norm convergence, we can further assume that $ f \in C_c^\infty(\He^n).$ Then, as a consequence of the uniform continuity of $ f $, we have 
$$\lim_{\epsilon \to 0} F(\epsilon w, \epsilon s) = \lim_{\epsilon \to 0} \left(\int_{\He^n} \left| f(z, t - \epsilon s - \frac{\epsilon}{2} \Im(z\cdot \bar{w})) - f(z,t) \right|^p dz \, dt \right)^{1/p} = 0. $$
Now, since $ F(\vep w,\vep s) \leq 2 \| f\|_p $ and $K_\gamma$ is integrable, it follows from the dominated convergence theorem that 
$$\lim_{\epsilon \to 0} \int_{\He^n}  K_{\gamma}(w,s) F(\vep w,\vep s) \, dw \, ds = 0. $$
This completes the proof of Proposition \ref{fractional-sublap-convergence}.
\end{proof}

The above arguments will fail in the case of the  operators $ \tilde{X}_j^\vep (\gamma^2+\tilde{\mathcal{L}}_\vep)^{-1/2} $ since the kernels of these are no longer integrable. However, as a corollary to the above proposition we can easily deduce the following result.

\begin{corollary} \label{schwarz-pointwise}
Let $ \gamma >0 $ and assume that $ f $ is a Schwartz class function  on $ \He^n.$ Then $ \tilde{X}_j^\vep (\gamma^2+\tilde{\mathcal{L}}_\vep)^{-1/2}f(z,t) $ converges pointwise to $ \frac{\partial}{\partial x_j}(\gamma^2-\Delta_{\C^n})^{-1/2}f(z,t) $ as $ \vep  \rightarrow 0.$ Same is true when  $ \tilde{X}_j^\vep$ are replaced by $ \tilde{Y}_j^\vep.$
\end{corollary}
\begin{proof} In view of the explicit formula (\ref{vfiel}) for the left invariant vector fields $ X_j^\vep $ and the corresponding formula for $\tilde{X}_j^\vep$ we have the relation $ \tilde{X}_j^\vep = 
X_j^\vep - \vep y_j \frac{\partial}{\partial t}.$ Therefore,
$$ \tilde{X}_j^\vep( K_{\gamma,\vep} \ast_\vep f)(z,t) = \left( K_{\gamma,\vep} \ast_\vep X_j^\vep f \right) (z,t) - \vep y_j \left( K_{\gamma,\vep}\ast_\vep \frac{\partial}{\partial t} f \right) (z,t) $$
where we have used the fact that for any left invariant vector field $ X $ on $ \He_\varepsilon^n $ one has $ X( g \ast_\vep h) = g \ast_\vep (Xh).$ Using the explicit formula (\ref{vfiel}) we see that 
$ \tilde{X}_j^\vep( K_{\gamma,\vep} \ast_\vep f)(z,t)$ is given by
$$ \left( K_{\gamma,\vep} \ast_\vep \frac{\partial}{\partial x_j} f \right) (z,t) + \frac{\vep}{2} \left( K_{\gamma,\vep} \ast_\vep y_j f \right) (z,t) - \vep y_j \left( K_{\gamma,\vep}\ast_\vep \frac{\partial}{\partial t} f \right) (z,t) .$$
The result then follows immediately from the above proposition.
\end{proof}

As we are interested in the Riesz transforms $ R_j^\vep $ we would like to obtain them as limits of the operators $ \tilde{X}_j^\vep (\gamma^2+\tilde{\mathcal{L}}_\vep)^{-1/2} $ as $ \gamma \rightarrow 0.$ The following lemma shows that this is true at least in $ L^2(\He^n).$

\begin{lemma} \label{lem-gamma-riesz-transform}
For any $ f \in L^2(\He^n) $ the functions $ \tilde{X}_j^\vep (\gamma^2+\tilde{\mathcal{L}}_\vep)^{-1/2} f$ converge to $ R_j^\vep f $ in $ L^2(\He^n) $ as $ \gamma \rightarrow 0.$ 
\end{lemma}
\begin{proof} 
In view of the relations (\ref{rel-1}) and (\ref{rel-2}) we see that
$$ \left( \tilde{X}_j^\vep (\gamma^2+\tilde{\mathcal{L}}_\vep)^{-1/2}f\right)  \circ \varphi_\vep = \tilde{X}_j (\gamma^2/\vep+\tilde{\mathcal{L}})^{-1/2}( f \circ \varphi_\vep). $$ 
As we also have the relation  $ (R_j^\vep f) \circ \varphi_\vep  = R_j(f\circ \varphi_\vep) $ is is enough to prove the lemma when $ \vep =1.$ On the Fourier transform side, the operator $ R_{j,\gamma}= \tilde{X}_j (\gamma^2+\tilde{\mathcal{L}})^{-1/2} $ corresponds to the left multiplier $  m_{j,\gamma}(\lambda) = \frac{i}{2} \left( A_j(\lambda) + A_j(\lambda)^\ast \right) ( \gamma^2 + H(\lambda))^{-1/2} $ and $ R_j $ corresponds to $ m_j(\lambda) = \frac{i}{2} \left( A_j(\lambda) + A_j(\lambda)^\ast \right) H(\lambda)^{-1/2}.$ In view of the Plancherel theorem for the group Fourier transform on $ \He^n$ we see that 
$$ \| R_{j,\gamma}f - R_j f \|_2^2 = C_n \int_{-\infty}^\infty  \| (m_{j,\gamma}(\lambda)- m_j(\lambda))\widehat{f}(\lambda) \|_{HS}^2 |\lambda|^n d\lambda.$$ As  both $ T $ and $ T^\ast $ have the same Hilbert-Schmidt operator norm, we can calculate the right hand side of the above using the Hermite basis $ \Phi_\alpha^\lambda, \alpha \in \N^n $ obtaining
$$ \| R_{j,\gamma}f - R_j f \|_2^2 = C_n \int_{-\infty}^\infty   \left( \sum_{\alpha \in \N^n} \| \widehat{f}(\lambda)^* (m_{j,\gamma}(\lambda)^\ast- m_j(\lambda)^\ast) \Phi_\alpha^\lambda \|_2^2\right)   |\lambda|^n d\lambda.$$ 
Since $ H(\lambda) \Phi_\alpha^\lambda =  (2|\alpha|+n)|\lambda| \Phi_\alpha^\lambda $, \, $ A_j(\lambda)^\ast \Phi_\alpha^\lambda = (2\alpha_j+2)^{1/2} |\lambda|^{1/2} \Phi_{\alpha+e_j}^\lambda $ and $ A_j(\lambda) \Phi_\alpha^\lambda = (2\alpha_j)^{1/2} |\lambda|^{1/2} \Phi_{\alpha - e_j}^\lambda $, we have
\begin{align*} 
(m_{j,\gamma}(\lambda)^\ast- m_j(\lambda)^\ast) \Phi_\alpha^\lambda &= \frac{i}{2} \left( \left( \frac{(2 \alpha_j + 2)|\lambda|}{\left( \gamma^2 + (2|\alpha|+n+2)|\lambda| \right)} \right)^{\frac{1}{2}} - \left( \frac{2 \alpha_j + 2}{2|\alpha|+n+2} \right)^{\frac{1}{2}} \right) \Phi_{\alpha+e_j}^\lambda \\ 
& \quad + \frac{i}{2} \left( \left( \frac{(2 \alpha_j)|\lambda|}{\left( \gamma^2 + (2|\alpha|+n-2)|\lambda| \right)} \right)^{\frac{1}{2}} - \left( \frac{2 \alpha_j}{2|\alpha|+n-2} \right)^{\frac{1}{2}} \right) \Phi_{\alpha - e_j}^\lambda 
\end{align*}
and therefore 
\begin{align*}
& \| \widehat{f}(\lambda)^* (m_{j,\gamma}(\lambda)^\ast- m_j(\lambda)^\ast) \Phi_\alpha^\lambda \|_2 \\ 
& \leq \frac{1}{2} \left| \left( \frac{(2 \alpha_j + 2)|\lambda|}{\left( \gamma^2 + (2|\alpha|+n+2)|\lambda| \right)} \right)^{\frac{1}{2}} - \left( \frac{2 \alpha_j + 2}{2|\alpha|+n+2} \right)^{\frac{1}{2}} \right| \| \widehat{f}(\lambda)^* \Phi_{\alpha+e_j}^\lambda \|_2 \\ 
& \quad + \frac{1}{2} \left| \left( \frac{(2 \alpha_j)|\lambda|}{\left( \gamma^2 + (2|\alpha|+n-2)|\lambda| \right)} \right)^{\frac{1}{2}} - \left( \frac{2 \alpha_j}{2|\alpha|+n-2} \right)^{\frac{1}{2}} \right| \| \widehat{f}(\lambda)^* \Phi_{\alpha - e_j}^\lambda \|_2. 
\end{align*}
As the above coefficients go to zero as $ \gamma $ goes to zero, the result follows from Fubini-Tonelli theorem.
\end{proof}

We also require the following lemma which allows us to compare $ R_{j,\gamma}^\vep $ with the Riesz transforms $ R_j^\vep.$

\begin{lemma} For any $ \vep, \gamma > 0 $ the operators $ \tilde{\mathcal{L}}_\vep^{1/2} (\gamma^2+\tilde{\mathcal{L}}_\vep)^{-1/2} $ are uniformly bounded on $ L^p(\He^n) $ for any $ 1 \leq p \leq \infty.$

\end{lemma}
\begin{proof}
As in the case of the previous lemma, we also have the following relation
$$ \left( \tilde{\mathcal{L}}_\vep^{1/2} (\gamma^2+\tilde{\mathcal{L}}_\vep)^{-1/2}f\right)  \circ \varphi_\vep = \tilde{\mathcal{L}}^{1/2} (\gamma^2/\vep+\tilde{\mathcal{L}})^{-1/2}( f \circ \varphi_\vep). $$
Therefore, we assume that $ \vep = 1$ and consider $ \tilde{\mathcal{L}}^{1/2} (\gamma^2+\tilde{\mathcal{L}})^{-1/2} $ which we rewrite as  $ \left( 1- \gamma^2(\gamma^2+\tilde{\mathcal{L}})^{-1}\right) ^{1/2}. $
Once again, we make use of the power series expansion \eqref{expansion-half-power}, which says that $ (1-d)^{1/2} = 1- \sum_{k=1}^\infty c_k d^k$. Note also that this series converges for all $ |d| \leq 1.$ Therefore, we have the expansion
\begin{equation} \label{exp} \tilde{\mathcal{L}}^{1/2} (\gamma^2+\tilde{\mathcal{L}})^{-1/2} = I - \sum_{k=1}^\infty c_k \gamma^{2k} (\gamma^2+\tilde{\mathcal{L}})^{-k}.
\end{equation}
As the heat semigroup $ e^{- \eta \tilde{\mathcal{L}}} $ are uniformly bounded on $ L^p(\He^n), 1 \leq p \leq \infty  $ with norm one, it follows that
$$ \gamma^{2k} \| (\gamma^2+\tilde{\mathcal{L}})^{-k}f \|_p \leq \frac{\gamma^{2k}}{\Gamma(k)} \int_0^\infty \eta^{k-1} e^{- \eta \gamma^2} \| e^{- \eta \tilde{\mathcal{L}}}f\|_p \, d\eta \leq \|f\|_p.$$
The lemma now follows from the expansion (\ref{exp}) and the above estimates.
\end{proof}

\begin{remark} \label{preserve}
From the above lemma it follows that $ \tilde{\mathcal{L}}^{1/2} (\gamma^2+\tilde{\mathcal{L}})^{-1/2} $ preserves the Sobolev spaces $ W^{1,p}(\He^n)$ and $ \mathring{W}^{1,p}(\He^n)$. Indeed, as it  is a left Fourier multiplier, we have
$$  X (\tilde{\mathcal{L}}^{1/2} (\gamma^2+\tilde{\mathcal{L}})^{-1/2}f) =  \tilde{\mathcal{L}}^{1/2} (\gamma^2+\tilde{\mathcal{L}})^{-1/2} X f $$
for any left invariant vector field $ X.$ Hence, the above lemma gives
$$ \| X (\tilde{\mathcal{L}}^{1/2} (\gamma^2+\tilde{\mathcal{L}})^{-1/2}f) \|_p \leq C \| Xf\|_p .$$
\end{remark}

We are now in a position to prove the following result on the boundedness of $R_j $ on  homogeneous as well as non-homogeneous Sobolev spaces.

\begin{theorem}\label{Riesz} 
For any $ j =1, 2, \ldots, 2n$, the Riesz transforms $ R_j $ are not bounded on $ W^{1,1}(\He^n)$ or $ \mathring{W}^{1,1}(\He^n)$. The same is true for any higher order Riesz transform $ R_P $ associated to bi-graded solid harmonics.
\end{theorem} 
\begin{proof} We prove this theorem by contradiction. Under the assumption that $ R_j $ is bounded on $ W^{1,1}(\He^n)$ or  $ \mathring{W}^{1,1}(\He^n)$ we will show that the Euclidean Riesz transform $ R_j^0 $ on $ \C^n $ is bounded on $  \mathring{W}^{1,1}(\C^n).$ The theorem then follows from Theorem \ref{thm-BP}. 

In view of Remark \ref{preserve}, we observe that the operators $ R_{j,\gamma} = R_j  \tilde{\mathcal{L}}^{1/2} (\gamma^2+\tilde{\mathcal{L}})^{-1/2} $ are uniformly bounded on $  \mathring{W}^{1,1}(\He^n)$ (resp. $ W^{1,1}(\He^n) $). 
Since we have
$$  (R_{j,\gamma}^\vep f) \circ \varphi_\vep =  R_{j,\gamma/{\sqrt{\vep}}}(f \circ \varphi_\vep) $$
we claim  that $ R_{j,\gamma}^\vep $ are uniformly bounded on $  \mathring{W}^{1,1}(\Hep)$ provided $ R_{j,\gamma} $ are uniformly bounded on $ \mathring{W}^{1,1}(\He^n)$.  To see this, apply  any of the vector fields $ X_k^\vep $ or $ Y_k^\vep $ to the above and make use of (\ref{rel-1}) to get 
$$  (X_k^\vep R_{j,\gamma}^\vep f) \circ \varphi_\vep =  \sqrt{\vep} X_kR_{j,\gamma/{\sqrt{\vep}}}(f \circ \varphi_\vep) =  R_{j,\gamma/{\sqrt{\vep}}}( (X_k^\vep f) \circ \varphi_\vep) $$ 
and a similar expression for the application of $ Y_k^\vep.$ The uniform boundedness of $ R_{j,\gamma} $ gives
$$ \int_{\He^n} |(X_k^\vep R_{j,\gamma}^\vep f) \circ \varphi_\vep(g) | \, dg \leq C \sqrt{\vep} \sum_{k=1}^n \int_{\He^n} \left( |X_k(f\circ \varphi_\vep)(g) |+  |Y_k(f\circ \varphi_\vep)(g)|\right) dg.$$
Once again using (\ref{rel-1}) we see that the right hand side reduces to
$$ C \sum_{k=1}^n \int_{\He^n} \left( | (X_k^\vep f)\circ \varphi_\vep(g) | + |(Y_k^\vep f)\circ \varphi_\vep(g)|\right) dg$$
and this proves our claim. We take a Schwartz function $ f $  and  $ h \in L^\infty(\He^n)$ and consider
$$ \int_{\He^n} X_k^\vep R_{j,\gamma}^\vep f(g) h(g) \, dg  = \int_{\He^n}  R_{j,\gamma}^\vep(X_k^\vep f)(g) h(g) \, dg. $$
By what we have proved above we obtain 
$$ \left| \int_{\He^n} R_{j,\gamma}^\vep(X_k^\vep f)(g) h(g) \, dg \right| \leq C \|h\|_\infty \sum_{k=1}^n \int_{\He^n} \left( | X_k^\vep f (g) |+  |Y_k^\vep f(g)|\right) dg.$$ 
Now we pass to the limit as $ \vep \rightarrow 0 $ and make use of the result of Corollary \ref{schwarz-pointwise} to conclude 
$$ \left| \int_{\He^n}  \frac{\partial}{\partial x_k} R_{j,\gamma}^0 f(z,t) h(z,t) \, dz \, dt \right| \leq C \|h\|_\infty \sum_{k=1}^n \int_{\He^n} \left( \left| \frac{\partial f}{\partial x_k}(z,t) \right| + \left| \frac{\partial f}{\partial y_k}(z,t) \right| \right) dz \, dt.$$ 
To complete the proof we remark that an analogue of Lemma \ref{lem-gamma-riesz-transform} is true for $ R_{j,\gamma}^0 ,$ the proof of which is much simpler. Hence as $ f \in L^2(\He^n) $ we can take the limit as $ \gamma \rightarrow 0 $ along a subsequence to conclude that
$$ \left| \int_{\He^n}  \frac{\partial}{\partial x_k} R_{j}^0 f(z,t) h(z,t) \, dz \, dt \right| \leq C \|h\|_\infty \sum_{k=1}^n \int_{\He^n} \left( \left| \frac{\partial f}{\partial x_k}(z,t) \right| + \left|\frac{\partial f} {\partial y_k}(z,t) \right| \right) dz \, dt.$$ 
This proves that $ R_j^0 $ are bounded on the homogeneous Sobolev space $ \mathring{W}^{1,1}(\C^n) .$

In the case of $ W^{1,1}(\He^n) $ we argue as follows. Assuming that $ R_{j,\gamma} $ are uniformly bounded on $ W^{1,1}(\He^n) $, we get that 
\begin{align*}
& \int_{\He^n} |(X_k^\vep R_{j,\gamma}^\vep f) \circ \varphi_\vep(g) | \, dg \\ 
& = \sqrt{\vep} \int_{\He^n} \left| X_k R_{j,\gamma/{\sqrt{\vep}}}(f \circ \varphi_\vep) (g) \right| dg \\ 
& \lesssim \sqrt{\vep} \left( \int_{\He^n} | f \circ \varphi_\vep(g) | \, dg + \sum_{k=1}^n \int_{\He^n} \left( |X_k(f\circ \varphi_\vep)(g) |+  |Y_k(f\circ \varphi_\vep)(g)|\right) dg \right) \\ 
& \lesssim \sqrt{\vep} \int_{\He^n} | f \circ \varphi_\vep(g) | \, dg + \sum_{k=1}^n \int_{\He^n} \left( | (X_k^\vep f)\circ \varphi_\vep(g) | + |(Y_k^\vep f)\circ \varphi_\vep(g)|\right) dg. 
\end{align*} 
Now, after applying the change of variables $\varphi_\vep^{-1}$, the term corresponding to the $L^1$-norm of $f$ will go to zero as $\epsilon \to 0$ because of the presence of the scalar multiple factor $\sqrt{\epsilon}$, and the remaining argument, as earlier, will help conclude that $ R_j^0 $ are bounded on the homogeneous Sobolev space $ \mathring{W}^{1,1}(\C^n) $, which is a contradiction in view of Theorem \ref{thm-BP}. This completes the proof.
\end{proof}

%%%%%%%%%%%%%%%%%%%%%%%%%%%%%%%%%%%%%%%%%%%%%%%%%%%%%%%%%%%%%%%%%%%%%%%%%
%%%%%%%%%%%%%%%%%%%%%%%%%%%%%%%%%%%%%%%%%%%%%%%%%%%%%%%%%%%%%%%%%%%%%%%%%

\subsection{Homogeneous multipliers on Sobolev spaces} \label{subsec-hom-mult-Sob} 
In this subsection we study the boundedness of  homogeneous (left) Fourier multipliers on $ W^{N,1}(\He^n) $ or $ \mathring{W}^{N,1}(\He^n) $ and prove Theorem \ref{thm2}. So we are assuming that $ T_m $  commutes with dilations and the multiplier $ m $ satisfies the condition (\ref{rep-hei}), viz
\begin{equation}\label{rep-heis} m_\delta(\lambda) = \sum_{j=1}^{d(\delta)} c_{\delta,j} G_\lambda(P^\delta_j) H(\lambda)^{-(a+b)/2}.\end{equation}
Though some of the following has been already discussed in the introduction, we recall them here for the sake  of readers' convenience. Let us recall that the unitary group $ U(n) $ has an action on $ \He^n$ given by $ R_\sigma (z,t) = (\sigma^{-1} z, t) $ which act as automorphisms of the group. Therefore, $ \pi_\lambda \circ R_\sigma $ are again irreducible unitary representations of $ \He^n $ and by the theorem of Stone-von Neumann it follows that $ \pi_\lambda \circ R_\sigma $ is unitarily equivalent to $ \pi_\lambda.$ Consequently, for any $ \lambda \in \R^\ast$ and $ \sigma \in U(n) $ we have a unitary operator $ \mu_\lambda(\sigma) $ acting on $ L^2(\R^n) $ such that
$$ \pi_\lambda \circ R_\sigma(z,t) = \mu_\lambda(\sigma)^\ast \pi_\lambda(z,t) \mu_\lambda(\sigma) $$
The action of $ U(n) $ on $ \He^n$  gives rise to an action on functions by $ R_\sigma f(z,t) = f(R_\sigma(z,t)) = f(\sigma^{-1}z,t).$  It is then easy to check  that 
$$ \widehat{R_\sigma f}(\lambda) = \mu_\lambda(\sigma) \widehat{f}(\lambda) \mu_\lambda(\sigma)^\ast .$$
As in the Euclidean case we have the following relation for Fourier multipliers on $ \He^n$: 
$$ R_\sigma \circ T_m \circ R_{\sigma^{-1}} f = T_{R_\sigma m}f $$
where $ R_\sigma m(\lambda) = \mu_\lambda(\sigma) m(\lambda) \mu_\lambda(\sigma)^\ast.$ For a class one representation $ \delta $ of $ U(n) $ realised on $ \mathcal{H}_\delta ,$ the space of bi-graded spherical harmonics of bi-degree $ (a,b),$ let  $ \chi_\delta $ stand for the character of $ \delta $ and define
\begin{equation} \label{eq-2-repeat} 
m_\delta(\lambda)  = \int_{U(n)}  \chi_{\delta}(\sigma^{-1}) \mu_\lambda(\sigma) m(\lambda) \mu_\lambda(\sigma)^\ast \, d\sigma .\end{equation} Then we have the following relation which is the analogue of  (\ref{eq-1}):
\begin{equation}
\int_{U(n)}  \chi_\delta(\sigma^{-1}) R_\sigma \circ T_m \circ R_{\sigma^{-1}} f \, d\sigma =  T_{m_\delta}f . 
\end{equation}

We are now in a position to prove Theorem \ref{thm2} stated in the introduction. 

\begin{proof}[Proof of Theorem \ref{thm2}] 
In proving Theorem \ref{thm2} we closely follow the proof of Theorem \ref{Main-Weyl} presented in Section \ref{sec-Weyl-mult-proof}. As in the case of Weyl multipliers we prove the theorem when $ N =1.$  First of all we have analogue of Lemma \ref{Weyl-rot}: when $ m $ is a Fourier multiplier for $ \mathring{W}^{1,1}(\He^n) $ (resp. $ W^{1,1}(\He^n) $) then so is $ m_\delta $ for any $ \delta \in \widehat{K_0}.$ This is a consequence of the relation \eqref{eq-2-repeat}. We can repeat the proof of Lemma \ref{Weyl-rot} verbatim as $ R_\sigma $ acts only on the $z $-variable.  The condition (\ref{rep-heis}) allows us to conclude that $ G_\lambda(P)H(\lambda)^{-(a+b)/2},$ where 
$ P = \sum_{j=1}^{d(\delta)} c_{\delta,j} P^\delta_j $, defines a bounded multiplier for $ \mathring{W}^{1,1}(\He^n) $ (resp. $ W^{1,1}(\He^n) $). But in view of Theorem \ref{Riesz} we can conclude that $ m_\delta = 0 $ for any $ \delta $ other than the trivial representation. (We can make use of  an analogue of Lemma \ref{lem-solid-harmonic-equiv} to reduce matters to the particular Riesz transform with multiplier $ (A_k(\lambda)^\ast)^b A_j(\lambda)^a H(\lambda)^{-(a+b)/2}$). Then we can make use of Theorem \ref{Riesz} to conclude that $ m(\lambda) = c(\lambda) I $ and finally the condition \eqref{hom} on $ m(\lambda) $ allows us to complete the proof.
\end{proof}

\begin{remark} It would be interesting to see if Theorem \ref{thm2} is true under the sole assumption \eqref{hom} without the extra condition (\ref{rep-hei}) imposed. From the relation (\ref{eq-2}) it follows that $ T_{m_\delta} $ commutes with the dilations $ \delta_r $ whenever $ m $ does. This is a consequence of the fact that $ R_\sigma $ commutes with the dilations. Consequently, we see that 
$$  m_\delta(\lambda)  = d_{\sqrt{|\lambda|}} \circ  m_\delta\left( \lambda / |\lambda| \right) \circ d_{\sqrt{|\lambda|}}^{-1} .$$
Thus for $ \lambda >0, m_\delta(\lambda) $ is completely determined by $ m_\delta(1) $ which can be expanded in terms of $ S^\rho_{j,k}, \rho \in \widehat{K_0}, 1 \leq j \leq d(\rho), k \in \N.$ From the definition of $ m_\delta(1) $ it follows that $ (m_\delta(1)P_k, S^\rho_{j,k}) = 0 $ unless $ \rho = \delta.$ Thus we have
$$  m_\delta(1) P_k = \sum_{j=1}^{d(\delta)} ( m_\delta(1), G(P^\delta_j))_k  C_\delta(2k+n)^{-2} G(P^\delta_j) P_k.$$
If we further assume that $ m(1) $ and hence $ m_\delta(1) $ are homogeneous of degree zero, then the above leads to the representation
$$ m_\delta(1) = \sum_{j=1}^{d(\delta)} B^\delta_j  G(P^\delta_j) C_\delta(H)^{-1} $$
where we have used the same notation as in Section \ref{sec-Weyl-mult-proof}. For any $ \lambda \in \R^\ast $ the above gives us
$$m_\delta(\lambda) = \sum_{j=1}^{d(\delta)} B^\delta_j  G_\lambda(P^\delta_j) C_\delta(H(\lambda))^{-1} .$$
In the case of Weyl transform, we have shown that $ C_\delta(H) H^{-(a+b)/2} $ define $ L^1(\C^n) $ bounded Weyl  multipliers. If we can show that $ C_\delta(H(\lambda))H(\lambda)^{-(a+b)/2} $ define $ L^1(\He^n) $ bounded Fourier multipliers, then we can drop the condition (\ref{rep-hei}) in Theorem \ref{thm2}. 
\end{remark}

%%%%%%%%%%%%%%%%%%%%%%%%%%%%%%%%%%%%%%%%%%%%%%%%%%%%%%%%%%%%%%%%%%%%%%%%%
%%%%%%%%%%%%%%%%%%%%%%%%%%%%%%%%%%%%%%%%%%%%%%%%%%%%%%%%%%%%%%%%%%%%%%%%%

\section*{Acknowledgements}
The authors would like to express their gratitude to Sayan Bagchi and Saurabh Shrivastava for several discussions and suggestions throughout the development of this project. 

The first author is supported by a scholarship from CSIR, Government of India. The second author is supported by INSPIRE Faculty Fellowship from DST, Government of India. The third author is supported by J. C. Bose Fellowship from DST, Government of India.
%%%%%%%%%%%%%%%%%%%%%%%%%%%%%%%%%%%%%%%%%%%%%%%%%%%%%%%%%%%%%%%%%%%%%%%%
%%%%%%%%%%%%%%%%%%%%%%%%%%%%%%%%%%%%%%%%%%%%%%%%%%%%%%%%%%%%%%%%%%%%%%%%
%%%%%%%%%%%%%%%%%%%%%%%%%%%%%%%%%%%%%%%%%%%%%%%%%%%%%%%%%%%%%%%%%%%%%%%%
%%%%%%%%%%%%%%%%%%%%%%%%%%%%%%%%%%%%%%%%%%%%%%%%%%%%%%%%%%%%%%%%%%%%%%%%
%%%%%%%%%%%%%%%%%%%%%%%%%%%%%%%%%%%%%%%%%%%%%%%%%%%%%%%%%%%%%%%%%%%%%%%%
%%%%%%%%%%%%%%%%%%%%%%%%%%%%%%%%%%%%%%%%%%%%%%%%%%%%%%%%%%%%%%%%%%%%%%%%
%%%%%%%%%%%%%%%%%%%%%%%%%%%%%%%%%%%%%%%%%%%%%%%%%%%%%%%%%%%%%%%%%%%%%%%%
%%%%%%%%%%%%%%%%%%%%%%%%%%%%%%%%%%%%%%%%%%%%%%%%%%%%%%%%%%%%%%%%%%%%%%%%
%%%%%%%%%%%%%%%%%%%%%%%%%%%%%%%%%%%%%%%%%%%%%%%%%%%%%%%%%%%%%%%%%%%%%%%%
%%%%%%%%%%%%%%%%%%%%%%%%%%%%%%%%%%%%%%%%%%%%%%%%%%%%%%%%%%%%%%%%%%%%%%%%
%%%%%%%%%%%%%%%%%%%%%%%%%%%%%%%%%%%%%%%%%%%%%%%%%%%%%%%%%%%%%%%%%%%%%%%%
%%%%%%%%%%%%%%%%%%%%%%%%%%%%%%%%%%%%%%%%%%%%%%%%%%%%%%%%%%%%%%%%%%%%%%%%

\end{document}